\documentclass[a4paper]{amsart}
\usepackage[utf8]{inputenc}

\usepackage[a4paper,margin=1in,headsep=25pt,footskip=35pt]{geometry}
\usepackage{bbm}
\usepackage{amsmath}
\usepackage{amsfonts}
\usepackage{amssymb}
\usepackage{amsthm}
\usepackage{enumerate}
\usepackage{color}
\usepackage{mathrsfs}
\usepackage{stmaryrd}
\usepackage{hyperref}
\usepackage{graphicx}
\usepackage{subcaption}
\SetSymbolFont{stmry}{bold}{U}{stmry}{m}{n}

\newtheorem{thm}{Theorem}

\newtheorem{lem}[thm]{Lemma}
\newtheorem{prop}[thm]{Proposition}

\newtheorem{defn}[thm]{Definition}
\newtheorem*{defn*}{Definition}
\newtheorem{assump}[thm]{Assumption}
\theoremstyle{remark}
\newtheorem{rem}[thm]{Remark}
\newtheorem{nrem}{Notational Remark}
\newtheorem{ex}{Example}

\newcommand{\deq}{\mathrel{\mathop:}=}

\newcommand{\R} {\mathbb{R}}
\newcommand{\C} {\mathbb{C}}
\newcommand{\D} {\mathbb{D}}

\newcommand{\N} {\mathbb{N}}

\newcommand{\adj}{^{*}} 
\newcommand{\tp}{^{\intercal}}

\newcommand{\dist} {\mathrm{dist}}

\DeclareMathOperator{\diag}{diag}
\DeclareMathOperator{\tr}{tr}
\DeclareMathOperator{\Tr}{Tr}

\DeclareMathOperator{\supp}{supp}
\DeclareMathOperator{\spec}{spec}

\DeclareMathOperator{\re}{\mathrm{Re}}
\DeclareMathOperator{\im}{\mathrm{Im}}

\newcommand{\caC}{{\mathcal C}}
\newcommand{\caD}{{\mathcal D}}

\newcommand{\caM}{{\mathcal M}}

\newcommand{\caS}{{\mathcal S}}

\newcommand{\bbD}{{\mathbb D}}

\newcommand{\bbH}{{\mathbb H}}

\newcommand{\bsa}{{\boldsymbol a}}
\newcommand{\bsb}{{\boldsymbol b}}

\newcommand{\bsd}{{\boldsymbol d}}
\newcommand{\bse}{{\boldsymbol e}}

\newcommand{\bsh}{{\boldsymbol h}}

\newcommand{\bsm}{{\boldsymbol m}}

\newcommand{\bss}{{\boldsymbol s}}

\newcommand{\bsu}{{\boldsymbol u}}
\newcommand{\bsv}{{\boldsymbol v}}

\newcommand{\bsx}{{\boldsymbol x}}

\newcommand{\bsG}{{\boldsymbol G}}
\newcommand{\bsH}{{\boldsymbol H}}

\newcommand{\bsM}{{\boldsymbol M}}

\newcommand{\wt}{\widetilde}
\newcommand{\ol}{\overline}

\newcommand{\beq}{ \begin{equation} }
	\newcommand{\eeq}{ \end{equation} }
\newcommand{\beqs}{\begin{equation*}}
	\newcommand{\eeqs}{\end{equation*}}
\newcommand{\baligned}{\beq\begin{aligned}}
	\newcommand{\ealigned}{\end{aligned}\eeq}

\newcommand{\lone}{\mathbbm{1}} 

\newcommand{\dd}{\mathrm{d}}
\newcommand{\ii}{\mathrm{i}}

\renewcommand{\P}{\mathbb{P}}

\newcommand{\AND}{\quad\text{and}\quad}

\newcommand\norm[1]{\Vert#1\Vert}
\newcommand\Norm[1]{\left\Vert#1\right\Vert}

\newcommand\Absv[1]{\left\vert#1\right\vert}
\newcommand\absv[1]{\vert#1\vert}

\newcommand\brkt[1]{\langle#1\rangle}
\newcommand\Brkt[1]{\left\langle#1\right\rangle}

\numberwithin{equation}{section} 
\numberwithin{thm}{section}
\numberwithin{ex}{section}

\title{Density of Brown measure of free circular Brownian motion}
\date{\today}
\author{L\'{a}szl\'{o} Erd\H{o}s} 
\author{Hong Chang Ji}
\address{Institute of Science and Technology Austria \\ Am Campus 1, 3400 Klosterneuburg, Austria}
\email{lerdos@ist.ac.at}
\address{Department of Mathematics, University of Wisconsin-Madison \\ 480 Lincoln Dr, Madison, WI 53706, United States}
\email{hji56@wisc.edu}
\thanks{$^\dagger$ Partially supported by ERC Advanced Grant "RMTBeyond" No. 101020331}
\thanks{$\ddagger$ Supported by ERC Advanced Grant "RMTBeyond" No. 101020331}

\begin{document}
	\begin{abstract}
	We consider the Brown measure of the free circular Brownian motion, $\bsa+\sqrt{t}\bsx$, with an arbitrary initial condition $\bsa$, i.e. $\bsa$ is a general non-normal operator and $\bsx$ is a circular element $*$-free from $\bsa$. 
	We prove that, under a mild assumption on $\bsa$, the density of the Brown measure has one of the following two types of behavior around each point on the boundary of its support 
	-- either (i) sharp cut, i.e. a jump discontinuity along the boundary, or (ii) quadratic decay at certain critical points on the boundary.
	Our result is in direct analogy with the previously known phenomenon for the spectral density of free semicircular Brownian motion, whose singularities are either a square-root edge or a cubic cusp.
	We also provide several examples and counterexamples, one of which shows that our assumption on $\bsa$ is necessary.
\end{abstract}
	\maketitle
	\section{Introduction}

Let $\caM$ be a von Neumann algebra with faithful, normal, tracial state $\brkt{\cdot}$. In \cite{Brown1986}, Brown proved that for every operator $\bsa\in\caM$ we may associate a probability measure $\rho_{\bsa}$ on $\C$ uniquely determined by
\beq\label{eq:Brown}
\brkt{\log\absv{\bsa-z}}=\int_{\C}\log\absv{w-z}\dd\rho_{\bsa}(w),\qquad \forall z\in\C,
\eeq
or equivalently
\beq\label{eq:Brown_fund}
\rho_{\bsa}=\frac{1}{2\pi}\Delta\brkt{\log\absv{\bsa-\cdot}},
\eeq
with the distributional Laplacian on the right-hand side. Obviously \eqref{eq:Brown_fund} implies that $\rho_{\bsa}$ is identical to the spectral measure when $\bsa$ is normal (so that $\bsa\bsa\adj=\bsa\adj\bsa$). It was also proved in \cite{Brown1986} that $\rho_{\bsa}$ is consistent with the holomorphic functional calculus.

The main purpose of our paper is to study the Brown measure of the sum $\bsa+\sqrt{t}\bsx$ around the edge, where $\bsa\in\caM$ is a general operator, $\bsx\in\caM$ is a circular element\footnote{A circular element is defined by the sum $(\bss_{1}+\ii\bss_{2})/\sqrt{2}$ where $\bss_{1},\bss_{2}$ are free pair of semicircular elements.} $*$-free from $\bsa$, and $t>0$.
If the time $t$ varies, the flow $t\mapsto \bsa+\sqrt{t}\bsx$ is also known as the \emph{free circular Brownian motion}. Several of our results may be compared with analogous known results in the Hermitian case, for the semicircular flow $\bsa+\sqrt{t}\bss$ where $\bsa=\bsa\adj$ and $\bss$ is a semicircular element.

Somewhat informally, our results are as follows. In the main result, Theorem~\ref{thm:sharp}, we identify the behavior of the Brown measure of $\bsa+\sqrt{t}\bsx$ at the boundary of its support. More specifically, we prove that if $\bsa$ satisfies (see Assumption \ref{assump})
\beq\label{eq:assu}
\Brkt{\frac{1}{\absv{\bsa-z}^{2}}}>\frac{1}{t}, \qquad \forall z\in\spec(\bsa),
\eeq
then at any point $z_{0}\in\partial\supp\rho_{\bsa+\sqrt{t}\bsx}$ the density $\rho$ of $\rho_{\bsa+\sqrt{t}\bsx}$ admits one of the following two asymptotic expansions around $z_{0}$; 
\begin{itemize}
	\item[(a)](Sharp edge) as $z\to z_{0}$, we have\footnote{The asymptotic relation $\sim$ in \eqref{eq:sharp} and \eqref{eq:quad} denotes that the left-hand sides are bounded from above and below by the right-hand sides, up to a constant factor that depends only on $\bsa$, $z_{0}$, and $t$. See also Notational Remark \ref{nrem:asymp}.}
	\beq\label{eq:sharp}
	\rho(z)\sim \lone(z\in\supp\rho_{\bsa+\sqrt{t}\bsx}).
	\eeq
	\item[(b)](Quadratic edge) as $z\to z_{0}$ while away from a certain direction in $\C\cong\R^{2}$, we have
	\beq\label{eq:quad}
	\rho(z)\sim \absv{z-z_{0}}^{2}\lone(z\in\supp\rho_{\bsa+\sqrt{t}\bsx}).
	\eeq
\end{itemize}
We here emphasize that our result is exhaustive, that is, no decay rate for $\rho$ other than zeroth or second order is possible at the boundary. Note that the first case (a) with sharp edge is exactly the behavior of $\rho_{\bsx}$, in which case the density is given by $\rho(z)=(\pi )^{-1}\lone(\absv{z}\leq1)$. The second case (b) may occur only on an analytic submanifold of the boundary (which is typically discrete), consisting of special \emph{critical points}. In fact, (b) does not happen except for countably many $t$'s; see Remark~\ref{rem:C} for details. We are not aware of any previous result that shows the second asymptotics (b) for general $\bsa$.

The assumption \eqref{eq:assu} is a mild, typical condition, which for example is true for limits of natural random matrix models; see Remark~\ref{rem:assump_typical}. In fact, in our other result Theorem~\ref{thm:higher}, we construct a concrete example showing that the assumption \eqref{eq:assu} is necessary for Theorem~\ref{thm:sharp}. In particular, we prove that the density may have an arbitrarily fast decay without \eqref{eq:assu}. 

Finally in Theorem \ref{thm:conn}, we prove for general $\bsa$ that the number of connected components of $\supp\rho_{\bsa+\sqrt{t}\bsx}\cup\spec(\bsa)$ decreases in $t>0$. The corresponding result for the Hermitian free Brownian motion $\bsa+\sqrt{t}\bss$ was proved in \cite[Proposition 3]{Biane1997}.

Theorem \ref{thm:sharp} is motivated by the corresponding Hermitian phenomenon proved in \cite{Alt-Erdos-Kruger2020}, whose result applies to the spectral measure of the free semicircular Brownian motion $\bsa+\sqrt{t}\bss$ with Hermitian $\bsa$ and semicircular element $\bss$. 
It was proved in \cite[Theorem 2.5]{Alt-Erdos-Kruger2020} that, when $\bsa$ is Hermitian satisfying a condition\footnote{The actual condition in \cite[Theorem 2.5]{Alt-Erdos-Kruger2020} is that $\sup_{\im z>0}\norm{\bsm(z)}<\infty$, where $\bsm$ solves the $\caM$-valued Dyson equation (see e.g. \cite[Eq. (2.3)]{Alt-Erdos-Kruger2020}). Specializing to $\bsa+\sqrt{t}\bss$, it is easy to see that the solution is $\bsm(z)=1/(\bsa-z-tm(z))$ where $m(z)$ is the Stieltjes transform of $\rho_{\bsa+\sqrt{t}\bss}$ (see \eqref{eq:Pastur_Herm}). Then \eqref{eq:assu} implies $\sup_{\im z>0}\norm{\bsm(z)}<\infty$ by \cite[Lemma 4]{Biane1997}. Essentially the same, but more quantitative proof that \eqref{eq:assu} implies $\sup_{\im z>0}\norm{\bsm(z)}<\infty$ appeared also in \cite[(A.4)]{Lee-Schnelli-Stetler-Yau2016}.\label{ftnt:m}}
that is implied by \eqref{eq:assu},
any spectral edge of $\bsa+\sqrt{t}\bss$ is either a square root edge or a cubic cusp. More precisely, if we write $\rho$ for the density of $\rho_{\bsa+\sqrt{t}\bss}$, then for each $x_{0}\in\partial\{x:\rho(x)>0\}$ we have either
\beq\label{eq:edge}
\rho(x)\sim \sqrt{(x-x_{0})_{\pm}} \qquad \text{as }x\to x_{0},
\eeq
where $(x-x_{0})_{\pm}$ can be the positive or negative part of $(x-x_{0})$ depending on whether $x_{0}$ is a left or right edge point, or
\beq\label{eq:cusp}
\rho(x)\sim \absv{x-x_{0}}^{1/3}\qquad \text{as }x\to x_{0}.
\eeq

Another closely related motivation comes from \emph{edge universality} for local eigenvalue statistics of random matrices. From \cite[Thoerem 6]{Sniady2002}, it is well-known that $\rho_{\bsa+\bsx}$ is the large $N$ limit of $\rho_{A+X}$, where $A+X\in\C^{N\times N}$ is a deformed random matrix with $A$ converging to $\bsa$ and $X$ consisting of i.i.d. entries with mean zero and variance $1/N$. Hence our result implies that the eigenvalue density of $A+X$ has universal (over $A$) macroscopic behavior at an edge point, and more importantly it identifies the two universality classes of local eigenvalue statistics corresponding to (a) and (b). Indeed, our results have already been used in very recent papers \cite{Liu-Zhang2023arXiv} and \cite{Liu-Zhang2024arXiv} to prove the edge universality when (a) and (b) holds, if $X$ is a Ginibre ensemble and $A$ satisfies certain restrictive assumptions e.g. normality. Also, the current result is used in our work in preparation \cite{Campbell-Cipolloni-Erdos-Ji2024}, where we prove edge universality at sharp edges (case (a)), in full generality i.e. for any matrix of the form $A+X$ where $A$ is arbitrary and $X$ has an i.i.d. entry distribution.

\subsection{Previous works}
As mentioned above, our paper is largely motivated by the corresponding Hermitian problem, i.e. studying the spectral measure of the free sum $\bsa+\sqrt{t}\bss$ of a Hermitian $\bsa$ with a semicircular element $\bss$. Hence we first focus on the previous work on $\rho_{\bsa+\sqrt{t}\bss}$. The law $\rho_{\bsa+\sqrt{t}\bss}$ is often referred to as the free additive convolution and denoted by $\rho_{\bsa}\boxplus\rho_{\sqrt{t}\bss}$. Rigorous analysis on this subject began by Biane in \cite{Biane1997}, who proved various regularity properties of $\rho_{\bsa+\sqrt{t}\bss}$ for fully general $\bsa$, including boundedness and H\"{o}lder continuity of its density as well as the number of connected components of its support. 
See also \cite{Bercovici-Wang-Zhong2023} for a generalization of Biane's result on H\"{o}lder continuity to free infinitely divisible laws.
Later the free convolution of two generic measures was studied in the series of papers \cite{Belinschi2006,Belinschi2008,Belinschi2014} by Belinschi, using the analytic subordination results proved in \cite{Belinschi-Bercovici2007} by Belinschi and Bercovici. Notably, it was proved that a general free additive convolution is also absolutely continuous in \cite[Theorem 4.1]{Belinschi2008}, and has bounded density under mild assumptions in \cite[Corollary 8]{Belinschi2014}.

On a finer scale than in the aforementioned works, there have been several results to probe the density of free convolution at spectral edges. More precisely, a typical goal in this context is to prove that the density decays as square root around the spectral edge, similarly to the semicircular law. In \cite{Lee-Schnelli-Stetler-Yau2016,Shcherbina2011a,Shcherbina2011b} (see e.g. \cite[Lemma 3.5]{Lee-Schnelli-Stetler-Yau2016}), it has been proved that the free convolution $\rho_{\bsa}\boxplus\rho_{\sqrt{t}\bss}$ has square-root decay at extremal edges as in \eqref{eq:edge} under the assumption that
\beq\label{eq:assu_Herm}
\int_{\R}\frac{1}{(y-x)^{2}}\dd\rho_{\bsa}(y)=\Brkt{\frac{1}{(\bsa-x)^{2}}}>\frac{1}{t},\qquad\forall x\in\supp\rho_{\bsa}=\spec(\bsa).
\eeq
Notice that \eqref{eq:assu_Herm} is exactly the same as \eqref{eq:assu} when $\bsa$ is Hermitian. 
More recently in \cite[Theorem 2.2]{Bao-Erdos-Schnelli2020JAM}, a similar result was proved for the free additive convolution of two general measures with power-law decay at an extremal edge. 
These results are confined to the extremal edges precisely due to possible cusps in \eqref{eq:cusp}. 
As mentioned above, it was proved in \cite{Alt-Erdos-Kruger2020} that under a somewhat weaker condition\footref{ftnt:m} than \eqref{eq:assu_Herm} the only other type of singularity of $\rho_{\bsa+\sqrt{t}\bss}$ than the square root edge is the cubic cusp. 
See \cite{Moreillon2022,Moreillon-Schnelli2022} for more results on inner singularities of general free additive convolution. We also mention that if the semicircular element $\bss$ is generalized to (semi-)circular elements in certain block forms, then other types of singularities may also occur, see \cite{Kolupaiev2021,Kruger-Renfrew2021}.

Now we collect related works on the non-Hermitian object, the Brown measure $\rho_{\bsa+\sqrt{t}\bsx}$, where $\bsx$ is circular. Biane and Lehner in \cite[Section 5]{Biane-Lehner2001} proposed a method to compute the density of $\rho_{\bsa+\sqrt{t}\bsx}$ and gave several examples. In the case of Hermitian $\bsa$, very recently in \cite[Theorem 1.1]{Ho-Zhong2023} Ho and Zhong managed to completely characterize the support of $\rho_{\bsa+\sqrt{t}\bsx}$ and prove that the density is bounded by $1/(\pi t)$. More importantly, they also discovered that $\rho_{\bsa+\sqrt{t}\bss}$ (on $\R$) is a pushforward of $\rho_{\bsa+\sqrt{t}\bsx}$ (on $\C$) via a map involving analytic subordination functions. Results of \cite{Ho-Zhong2023} have been extended in \cite{Ho2022} where $\bsx$ was replaced with a more general elliptic element. The key approach in \cite{Ho2022,Ho-Zhong2023} was the PDE techniques developed in \cite{Driver-Hall-Kemp2022}, that were also important in several other recent works on free \emph{multiplicative} Brownian motion (with unitary initial condition); see e.g. \cite{Hall-Ho2023} and references therein.

A more related line of works is \cite{Belinschi-Yin-Zhong2022,Bercovici-Zhong2022,Zhong2021} by Zhong et al., where $\bsa$ can be completely general, typically non-Hermitian or even non-normal. In these works, the same characterization of $\supp\rho_{\bsa+\sqrt{t}\bsx}$ as in \cite{Ho-Zhong2023} was extended to general $\bsa$ beyond Hermitian, and a slightly more explicit formula than \cite{Ho-Zhong2023} was proved for the density of $\rho_{\bsa+\sqrt{t}\bsx}$. See e.g. \cite[Theorem 7.10]{Belinschi-Yin-Zhong2022} for a rigorous statement. 
We emphasize that \cite{Belinschi-Yin-Zhong2022,Bercovici-Zhong2022,Zhong2021} used purely free probabilistic methods, effectively making the proof depend only on specific functionals of $\bsa$ (see $f_{\bsa,\eta}(z)$ in \eqref{eq:Gfn}). Hence these works were able to cover general non-normal $\bsa$, whereas PDE methods in \cite{Ho-Zhong2023} actually computed partial derivatives of the log determinant of $\bsa$ in \eqref{eq:Brown}, requiring $\bsa$ to be Hermitian.

There is another independent series of papers \cite{Bordenave-Capitaine2016,Bordenave-Caputo-Chafai2014} by  Bordenave et al.  covering $\rho_{\bsa+\sqrt{t}\bsx}$. These two works mainly concern eigenvalues of large non-Hermitian random matrix of the form $A+X$, where $A$ is a deterministic matrix and $X$ consists of i.i.d. entries. Since the eigenvalues are asymptotically distributed as $\rho_{\bsa+\sqrt{t}\bsx}$ as the matrix size increases (see \cite{Sniady2002}), it was vital therein to study the support of $\rho_{\bsa+\sqrt{t}\bsx}$. In \cite[Theorem 1.4]{Bordenave-Caputo-Chafai2014}, when $\bsa$ is normal, the same formulas for the support and density of $\rho_{\bsa+\sqrt{t}\bsx}$ as in (and prior to) \cite{Belinschi-Yin-Zhong2022,Bercovici-Zhong2022,Zhong2021} were proved. Then in \cite[Proposition 1.2]{Bordenave-Capitaine2016}, for general $\bsa$, a slightly stronger characterization for the support of $\rho_{\bsa+\sqrt{t}\bsx}$ was proved under the additional assumption that $\supp\rho_{\bsa+\sqrt{t}\bsx}=\spec(\bsa+\sqrt{t}\bsx)$.
We remark that all previous works on $\rho_{\bsa+\sqrt{t}\bsx}$ mentioned above focused on qualitative properties of $\rho_{\bsa+\sqrt{t}\bsx}$, such as its support and a general upper bound on its density. In contrast, our results concern precise asymptotics of the density of $\rho_{\bsa+\sqrt{t}\bsx}$ at the boundary in a typical situation.

Finally, sharp cutoff of the density at the edge as in \eqref{eq:sharp} was proved in \cite[Proposition 2.4]{Alt-Erdos-Kruger2018} under a different setting, with $\bsa=0$ but more general $\bsx$ corresponding to a large random square matrix $X$ whose entries have varying variances (in contrast to a circular element, the limit of a Ginibre matrix). Later in \cite[Theorem 2.5]{Alt-Kruger2021} this result was further generalized to $X$ with correlated entries but still with $\bsa=0$. However, the Brown measure is radially symmetric when $\bsa=0$ making the problem easier.

\subsection{Methods}
Our proof of Theorem \ref{thm:sharp} mainly involves detailed analysis of the Schwinger-Dyson equation associated to the Hermitization of $\bsa+\sqrt{t}\bsx$. The Hermitization is the operator-valued map defined by
\beq
\C\ni z\mapsto \begin{pmatrix} 0 & \bsa+\sqrt{t}\bsx-z\\(\bsa+\sqrt{t}\bsx-z)\adj & 0 \end{pmatrix}\in \C^{2\times 2}\otimes \caM,
\eeq
whose image is Hermitian as the name suggests. By the definition of the Brown measure, it suffices to study the family of spectral measures $\{\rho_{\absv{\bsa+\sqrt{t}\bsx-z}}:z\in\C\}$, which precisely matches that of the Hermitization up to symmetrization. Then the $*$-freeness of $\bsa$ and $\bsx$ gives rise to the Schwinger-Dyson equation for the resolvent of the Hermitization given by (see Proposition \ref{prop:Dyson})
\beq\label{eq:Dyson_intro}
M(\zeta)=\begin{pmatrix} -\left(\zeta+\frac{t}{2}\Tr M(\zeta)\right) & \bsa-z \\ (\bsa-z)\adj & -\left(\zeta+\frac{t}{2}\Tr M(\zeta)\right) \end{pmatrix}^{-1},\qquad \zeta\in\C,\im \zeta>0.
\eeq
Here, $M$ is the $\C^{2\times2}$-valued Stieltjes transform of $\bsa+\sqrt{t}\bsx-z$ given by
\beq
M(\zeta)\deq(\mathrm{Id}\otimes\brkt{\cdot})\left(\begin{pmatrix}-\zeta & \bsa+\sqrt{t}\bsx-z \\ (\bsa+\sqrt{t}\bsx-z)\adj & -\zeta\end{pmatrix}^{-1}\right).
\eeq
The main task in our proof is to study the Stieltjes transform $\Tr M(\ii\eta)$ when $z$ is close to the spectral edge and $\eta>0$ is small; see Lemma \ref{lem:v_asymp}. We remark that \cite{Zhong2021} first derived and used the equation \eqref{eq:Dyson_intro} for general $\bsa$, and likewise, our method is more related to \cite{Belinschi-Yin-Zhong2022,Bercovici-Zhong2022,Zhong2021} than \cite{Ho2022,Ho-Zhong2023}; see Section \ref{sec:prelim} for more details.

While the Dyson equation in \eqref{eq:Dyson_intro} played a crucial role also in \cite{Belinschi-Yin-Zhong2022,Bercovici-Zhong2022,Zhong2021}, we study it in a different regime.
The line of works \cite{Belinschi-Yin-Zhong2022,Bercovici-Zhong2022,Zhong2021} handled general $\bsa$, but only considered $z$ well inside the bulk (corresponding to $\im\Tr M(\ii\eta)\sim 1$ as $\eta\to 0$) or far outside of the spectrum ($\im\Tr M(\ii\eta) \lesssim\eta$).	
In contrast, to study the density $\rho(w)$ around an edge point $z$, one needs to take both $|w-z|$ and $\eta$ small. This poses a major difficulty since the Dyson equation is highly unstable in the joint limit $\absv{w-z},\eta\to0$. The precise behavior of the density along the boundary can only be detected if one takes the more involved limit $\eta\to 0$ first. We resolve this instability in Lemma \ref{lem:v_asymp} which is the main new technical part of our proof.

We also mention that the sharp edge phenomenon was covered earlier in \cite{Alt-Erdos-Kruger2018,Alt-Erdos-Kruger2021,Bourgade-Yau-Yin2014Edge} with somewhat more general $\bsx$, but only for $\bsa=0$. The Dyson equation depended only on $\absv{z}$ in these papers, so that the Brown measure was rotationally invariant. Consequently, the second phenomenon (b) was simply absent therein, and even the geometry of (a) was essentially one dimensional, i.e. one only needed to show that $\rho(|z|)$ has a sharp cut-off.
But with non-zero $\bsa$, we have to keep track of $z$ as a genuinely two-dimensional parameter in order to prove (a) and (b) along a general boundary (beyond circles) of $\supp\rho_{\bsa+\bsx}$.

\subsection{Organization}
In Section \ref{sec:result}, we rigorously define our model and state the main results. Section \ref{sec:ex} is devoted to examples and counterexamples, along with pictorial illustrations. In Section \ref{sec:prelim} we provide free probabilistic preliminary results. Finally in Section~\ref{sec:proof}, we prove the main result Theorem \ref{thm:sharp}. The remaining results, Theorems \ref{thm:higher} and \ref{thm:conn}, are proved respectively in Appendices \ref{append:irreg} and \ref{append:conn}.

\begin{nrem}
	For an operator $\bsa$ in a von Neumann algebra $\caM$, we define $\absv{\bsa}\deq (\bsa\bsa\adj)^{1/2}$, $\absv{\bsa}_{*}\deq (\bsa\adj\bsa)^{1/2}$, and 
	\beqs
	\re\bsa\deq\frac{\bsa+\bsa\adj}{2},\qquad\qquad \im\bsa\deq\frac{\bsa-\bsa\adj}{2\ii}.
	\eeqs
	For a von Neumann algebra $\caM$, we define $\bbH_{+}(\caM)\deq\{\bsa\in\caM:\im \bsa>0\}$ to be the open upper-half plane in $\caM$. When $\caM=\C$, we use the shorthand notation $\C_{+}=\bbH_{+}(\C)$. For each $n\in\N$, we define $M_{n}(\caM)\deq\C^{n\times n}\otimes \caM$ to be the $*$-algebra of $(n\times n)$ matrices over $\caM$.
\end{nrem}
\begin{nrem}
	We write $D(z,r)\deq\{w\in\C:\absv{w-z}<r\}$ for $w\in\C$ and $r>0$, and denote $\bbD\deq D(0,1)$. The integral with respect to the Lebesgue measure on $\C$ is denoted by 
	\beqs
	\int_{\C}f(z)\dd^{2}z.
	\eeqs
\end{nrem}
\begin{nrem}\label{nrem:asymp}
	For two functions $f,g$ on a set $I$ with $g\geq 0$, we write $f\lesssim g$ to denote that $\absv{f(i)}\leq Cg(i)$ for all $i\in I$ and a constant $C>0$ that do not depend on $i$. For $f,g\geq 0$, we write $f\sim g$ if $f\lesssim g$ and $g\lesssim f$. If $f,g$ are parametrized by multiple indices, say, $(i,j)$, then we write $f\lesssim_{j}g$ if $f(i,j)\leq C(j)g(i,j)$ for all $(i,j)$, where $C(j)>0$ does not depend on $i$. Likewise, we write $f\sim_{j}g$ if $f\lesssim_{j}g$ and $g\lesssim_{j}f$.
\end{nrem}

\section{Definitions and main results}\label{sec:result}
\begin{defn}\label{defn:model}
	Let $(\caM,\brkt{\cdot})$ be a $W\adj$-probability space, that is, $\caM$ is a unital von Neumann algebra and $\brkt{\cdot}$ is a faithful, tracial, normal state on $\caM$. Let $\bsx,\bsa\in\caM$ be a $*$-free pair such that $\bsx$ is a free circular element. 
\end{defn}
Recall that a collection of $*$-subalgebras $\{\caM_{i}:i\in I\}$ of $\caM$ is called $*$-free if, for any $\bsb_{1},\cdots,\bsb_{m}\in\caM$ with $\bsb_{j}\in \caM_{i_{j}}$, we have
\beq
\begin{cases}
	\brkt{\bsb_{j}}=0,\,\forall j\in\{1,\cdots,m\},\\
	i_{1}\neq i_{m},\,\,i_{j}\neq i_{j+1},\,\forall j\in\{1,\cdots,m-1\},\\
\end{cases}
\Longrightarrow 
\brkt{\bsb_{1}\bsb_{2}\cdots \bsb_{m}}=0.
\eeq
A collection of elements $\{\boldsymbol{c}_{i}:i\in I\}$ in $\caM$ is called $*$-free if the $*$-algebras generated by $\boldsymbol{c}_{i}$ are $*$-free. A free circular element in $\caM$ is defined by the sum $\bsx\deq(\bss_{1}+\ii\bss_{2})/\sqrt{2}$ where $(\bss_{1},\bss_{2})$ is a free pair of semicircular elements in $\caM$, that is, each $\bss_{i}$ is self-adjoint with spectral distribution given by
\beq
\brkt{f(\bss_{i})}=\int_{\R}f(x)\frac{\sqrt{4-x^{2}}}{2\pi}\lone_{[-2,2]}(x)\dd x,\qquad f\in C([-2,2]).
\eeq

\begin{defn}
	For each $\bsb\in\caM$, we denote the spectrum of $\bsb$ by $\spec(\bsb)$. We write $\rho_{\bsb}$ for the Brown measure of $\bsb$ with respect to $\brkt{\cdot}$ defined in \eqref{eq:Brown_fund}. For each $\eta>0$, we define the function $f_{\bsb,\eta}:\C\to(0,\infty)$ as
	\beq\label{eq:Gfn}
	f_{\bsb,\eta}(z)\deq\Brkt{\frac{1}{\absv{\bsb-z}^{2}+\eta^{2}}},
	\eeq
	and we further define $f_{\bsb}:\C\to(0,\infty]$ as
	\beq\label{eq:def_f}
	f_{\bsb}(z)\deq\Brkt{\frac{1}{\absv{\bsb-z}^{2}}}\equiv\lim_{\eta\to0}f_{\bsb,\eta}(z).
	\eeq
\end{defn}
\begin{rem}\label{rem:f_prop}
	The functions $f_{\bsb,\eta},f_{\bsb}$ have the following properties which can be proved with elementary calculations.
	\begin{enumerate}[(i)]
		
		\item $f_{\bsb}$ is real analytic in $\C\setminus\spec(\bsb)$.
		
		\item $f_{\bsb,\eta}$ is continuous for each $\eta>0$, and $f_{\bsb}$ is lower semi-continuous. In particular the set $\{z\in\C: f_{\bsb}(z)>c\}$ is open for each $c\in\R$.
		%
		\item $f_{\bsb}$ is strictly subharmonic in $\C\setminus\spec(\bsb)$, more precisely, for all $z\in\C\setminus\spec(\bsb)$ we have
		\beq\label{eq:subhar}
		\Delta f_{\bsb}(z)=4\frac{\partial ^{2}}{\partial z\partial\ol{z}}f_{\bsb}(z)=4\Brkt{\frac{1}{\absv{(\bsb-z)^{2}}^{2}}}>0.
		\eeq

		\item We can recover the measure $\rho_{\absv{\bsb-z}}$ from the function $\eta\mapsto f_{\bsb,\eta}(z)$ via moments of $\absv{\bsb-z}^{2}$. In particular, we can compute $\brkt{\log\absv{\bsb-\cdot}}$ and hence $\rho_{\bsb}$.
	\end{enumerate}
\end{rem}
\begin{rem}\label{rem:supp}
	If $\bsb$ is normal, we trivially have $\supp\rho_{\bsb}=\spec(\bsb)$. For a general $\bsb$ we have only $\supp\rho_{\bsb}\subset\spec(\bsb)$, as $z\mapsto\brkt{\log\absv{\bsb-z}}$ is harmonic in $\C\setminus\spec(\bsb)$. Although $\supp\rho_{\bsb}$ may be strictly smaller than $\spec(\bsb)$ in general, in most cases they are equal; see Remark \ref{rem:R} for more details.
\end{rem}

Before presenting our new results, we here record a previous result on the Brown measure of $\bsa+\sqrt{t}\bsx$ from \cite{Belinschi-Yin-Zhong2022,Zhong2021}:
\begin{thm}[{\cite[Theorem 7.10]{Belinschi-Yin-Zhong2022} and \cite[Theorem 4.2]{Zhong2021}}]\label{thm:prev} For each $t>0$, define the open domain $\caD_{t}\subset\C$ as
	\beq\label{eq:def_D}
	\caD_{t}\deq\left\{z\in\C:f_{\bsa}(z)>\frac{1}{t}\right\}.
	\eeq
	Then we have the following.
	\begin{itemize}
		\item[(i)] $\supp\rho_{\bsa+\sqrt{t}\bsx}=\ol{\caD}_{t}$.
		
		\item[(ii)] $\rho_{\bsa+\sqrt{t}\bsx}$ is absolutely continuous with respect to the Lebesgue measure in $\C$.
		
		\item[(iii)] The density of $\rho_{\bsa+\sqrt{t}\bsx}$ is bounded by $(\pi t)^{-1}$ in $\C$, and positive, real analytic in $\caD_{t}$.
		
	\end{itemize}
\end{thm}
Recall from Remark \ref{rem:f_prop} (ii) that the domain $\caD_{t}$ is open, and it is always non-empty since $f_{\bsa}(z)=\infty$ except for $z$ in a $\rho_{\bsa}$-null set; see \eqref{eq:assump_check}--\eqref{eq:Lp} for a proof.

\begin{rem}
	
	Recall from \cite{Biane1997} that, for a free pair of Hermitian operators $(\bsb,\bss)$ in $\caM$ with a semicircular element $\bss$, the set $\supp(\rho_{\bsb+\sqrt{t}\bss})$ may not be increasing in $t$. In contrast, $\ol{\caD}_{t}$, the support of $\rho_{\bsa+\sqrt{t}\bsx}$, is always increasing in $t$ since $\caD_{t}$ increases by its definition. Also by \cite[Theorem 4.5]{Zhong2021} we have 
	\beq\label{eq:supp<D}
	\supp\rho_{\bsa}\subset \ol{\caD}_{t},\qquad \forall t>0.
	\eeq
\end{rem}

\subsection{Results}

We now present the new results of this paper. The main result, on the spectral edge of $\rho_{\bsa+\sqrt{t}\bsx}$, is obtained under the following mild regularity assumption on $\bsa$. Later in Section \ref{sec:assump_check}, we discuss how to check its validity in certain cases.
\begin{assump}\label{assump}
	 Fix  $t>0$. We assume $\spec(\bsa)\subset \caD_{t}$, that is, we have
	\beq\label{eq:assump_main}
	f_{\bsa}(z)=\Brkt{\frac{1}{\absv{\bsa-z}^{2}}}>\frac{1}{t},\qquad \forall z\in\spec(\bsa).
	\eeq
\end{assump}
\begin{rem}[Ubiquity of \eqref{eq:assump_main}]\label{rem:assump_typical}
	Observe that Assumption \ref{assump} automatically follows if both $\supp\rho_{\bsa}=\spec(\bsa)$ and $\supp\rho_{\bsa}\subset\caD_{t}$ hold true, that are slightly stronger than the trivially valid inclusions $\supp\rho_{\bsa}\subset\spec(\bsa)$ and $\supp\rho_{\bsa}\subset\overline{\caD}_{t}$ from Remark \ref{rem:supp} and \eqref{eq:supp<D}. While there are somewhat pathological counterexamples to $\supp\rho_{\bsa}=\spec(\bsa)$ or $\supp\rho_{\bsa}\subset\caD_{t}$ (see Remark \ref{rem:R} and Theorem \ref{thm:sharp}), one can still think of both as typical situations. For example, the first equality typically holds for $*$-limits of random matrices, and the second inclusion is satisfied unless $\rho_{\bsa}$ has a fast decay at some point in the support; see Remark \ref{rem:R} and \eqref{eq:Lp}, respectively. 
	We also remark that the first part $\supp\rho_{\bsa}=\spec(\bsa)$ can be directly compared with \cite[Assumption (A3)]{Bordenave-Capitaine2016}, which reads as $\supp\rho_{\bsa+\sqrt{t}\bsx}=\spec(\bsa+\sqrt{t}\bsx)$.
\end{rem}
\begin{rem}[Regularity of $\partial\caD_{t}$]\label{rem:assump_conseq}
	Note that $\spec(\bsa)\subset\caD_{t}$ implies the strict inclusion $\spec(\bsa)\subsetneq\caD_{t}$ since $\spec(\bsa)$ is compact and $\caD_{t}$ is open. 
	Thus, recalling the function $f_{\bsa}$ is real analytic (hence continuous) and strictly subharmonic  on $\C\setminus\spec(\bsa)$, we find that Assumption 7 implies
	\beq\label{eq:bd=level}
	\partial\caD_{t}=\left\{z\in\C:f_{\bsa}(z)=\frac{1}{t}\right\}.
	\eeq
	Then it also immediately follows that $\partial\caD_{t}$ is an analytic curve (see also Remark \ref{rem:C}), hence of Lebesgue measure zero in $\C$.
\end{rem}

\begin{thm}\label{thm:sharp}
	Let $t>0$ be fixed and $\bsa$ satisfy Assumption \ref{assump}. Define the set of critical points
	\beq\label{eq:def_C}
	\caC_{t}\deq\left\{z\in\partial\caD_{t}:\nabla f_{\bsa}(z)=0\right\},
	\eeq
	where $\nabla$ is the usual gradient in $\C\cong\R^{2}$.  Then there exists a density $\rho$ for $\rho_{\bsa+\sqrt{t}\bsx}$ that satisfies the following.
	\begin{itemize}
		\item[(i)] {\rm(Outside)} $\rho$ is identically zero on $\C\setminus\ol{\caD_{t}}=\{z\in\C:f_{\bsa}(z)<1/t\}$.
		\item[(ii)] {\rm(Bulk)} For each $\delta>0$, there exists a constant $c_{0}>0$ such that
		\beq\label{eq:sharp_bulk}
		\rho(z)\geq c_{0} \qquad \forall z\in\{z\in\caD_{t}:\dist(z,\partial\caD_{t})>\delta\}.
		\eeq
		
		\item[(iii)] {\rm(Sharp edge)} Let $z\in\partial\caD_{t}\setminus\caC_{t}$. The density $\rho(w)$ satisfies the following asymptotics as $w\to z$ in~$\caD_{t}$;
		\beq\label{eq:edge_sharp}
		\rho(w)=\frac{1}{4\pi}\brkt{\absv{\bsa-z}^{-4}}^{-1}\absv{\nabla f_{\bsa}(z)}^{2}+O(\absv{w-z}).
		\eeq
		In particular, there exist constants $c_{1},C_{1},\delta_{1}>0$ such that
		\beq\label{eq:edge_sharp_1}
		c_{1}\lone_{\ol{\caD}_{t}}(w)\leq \rho(w)\leq C_{1}\lone_{\ol{\caD}_{t}}(w),\qquad  \absv{w-z}<\delta_{1}.
		\eeq
		\item[(iv)] {\rm(Quadratic edge)} Let $z\in\caC_{t}$. The $(2\times 2)$ Hessian matrix $H[f_{\bsa}]$ of $f_{\bsa}$ (viewed as a function in $\R^{2}\cong\C$) at $z$ satisfies  $\Tr H[f_{\bsa}](z)>0$,  and the density $\rho(w)$ satisfies the following quadratic asymptotics as $w\to z$ in $\caD_{t}$;
		\beq\label{eq:edge_quad}
		\rho(w)= Q_{z}[w-z]+O(\absv{w-z}^{3}),
		\eeq
		where $u\mapsto Q_{z}[u]$ is a real-valued quadratic form on $\C$ defined by\footnote{In \eqref{eq:quad_def}, we identified $u\in\C$ with $(\re u,\im u)^{\tp}\in\R^{2}$ and $H[f_{\bsa}](z)$ with the corresponding $2\times 2$ real matrix. This choice is purely cosmetic as $\brkt{u,H[f_{\bsa}](z)u}$ and $\norm{H[f_{\bsa}](z)u}$ do not depend on the choice of an orthonormal basis of $\C\cong\R^{2}$.}
		\beq\label{eq:quad_def}
		Q_{z}[u]\deq \frac{1}{2\pi}\frac{\brkt{\absv{\bsa-z}^{-2}\absv{\bsa-z}_{*}^{-2}}}{\brkt{\absv{\bsa-z}^{-4}}}\brkt{u,H[f_{\bsa}](z)u}
		+\frac{1}{4\pi}\frac{1}{\brkt{\absv{\bsa-z}^{-4}}}\norm{H[f_{\bsa}](z)u}^{2}.
		\eeq
		Consequently, for any $\kappa\in(0,1)$ there exist constants $c_{2},C_{2},\delta_{2}>0$ such that
		\beq\label{eq:edge_quad_1}\begin{aligned}
			\rho(w)&\geq c_{2}\lone_{\caD_{t}}(w)\absv{w-z}^{2}, &
			\qquad \absv{w-z}&<\delta_{2},\,(w-z)\in\caS(z,\kappa) \\ 
			\rho(w)&\leq C_{2}\lone_{\caD_{t}}(w)\absv{w-z}^{2}, &
			\qquad \absv{w-z}&<\delta_{2}
		\end{aligned}\eeq
		where $\caS(z,\kappa)$ is the angular sector (see Figure \ref{fig:tang}) defined by
		\beq
		\caS(z,\kappa)\deq \left\{w\in\C\cong\R^{2}:\frac{\norm{P_{z}w}^{2}}{\norm{w}^{2}}<1-\kappa \right\}
		\eeq
		and $P_{z}\in\R^{2\times 2}$ denotes the orthogonal projection onto the null space of $H[f_{\bsa}](z)$. Note that if $H[f_{\bsa}](z)$ is not singular then $\caS(z,\kappa)=\C$ for any $\kappa\in(0,1)$.
	\end{itemize}
	The constant $c_{0}$ in (ii) depends only on $\bsa$ and $t$. The constants $c_{1},C_{1},C_{2},\delta_{1},\delta_{2}$ and the implicit constants in \eqref{eq:edge_sharp} and \eqref{eq:edge_quad} depend only on $\bsa$ and $z$. Finally, $c_{2}$ depends only on $\bsa,z$, and $\kappa$. Note that given $\bsa$ and $z\in\partial\caD_{t}$ we can recover $t$ via $1/t=\brkt{\absv{\bsa-z}^{-2}}$.
\end{thm}
We prove Theorem \ref{thm:sharp} in Section \ref{sec:proof}.
An important feature of Theorem \ref{thm:sharp} is that $\rho$ typically has a jump at the edge of its support, but at special points, when $\rho$ decays at the edge, only the quadratic rate is possible (at least within a cone). At first glance, the Taylor expansion in \eqref{eq:edge_quad} might seem to indicate that a higher order decay for $\rho$ is also possible by choosing $\bsa$ carefully. However, under Assumption \ref{assump}, we will show that $\Tr H[f_{\bsa}](z)>0$ for any $z\in\partial\caD_{t}$ so that $H[f_{\bsa}](z)$ is never fully degenerate. Actually, a higher-order decay can happen without Assumption \ref{assump}; see Theorem \ref{thm:higher} below.

\begin{rem}[Atypicality of quadratic edges]\label{rem:C}
	The set $\caC_{t}$ of critical points is small in the following two senses. 
	Firstly, for each $t>0$, by {\L}ojasiewicz stratification theorem (see e.g. \cite[Section 3.2]{Hardt1975}) $\partial\caD_{t}$ is a compact analytic manifold of dimension at most $1$ (since $f_{\bsa}$ is non-constant by \eqref{eq:subhar}), and $\caC_{t}$ is an analytic submanifold of $\partial\caD_{t}$. 
	Consequently, both $\partial\caD_{t}$ and $\caC_{t}$ have finitely many connected components, and each component of $\caC_{t}$ is either a singleton (for being of lower dimension than $\partial\caD_{t}$) or identical to a whole component of $\partial\caD_{t}$. 
	Secondly, notice that the set of $t$'s for which $\caC_{t}\neq\emptyset$ is precisely the set of critical \emph{values} of the real analytic function $f_{\bsa}$, i.e. we may write the set of such $t$'s as
	\beq
	T_{\bsa}\deq \left\{t>0:\frac{1}{t}\in f_{\bsa}(\wt{\caC})\right\},\qquad \wt{\caC}\deq \bigcup_{t>0}\caC_{t}=\{z\in\C\setminus\spec(\bsa):\nabla f_{\bsa}(z)=0\}.
	\eeq
	The set $\wt{\caC}$ of critical points is a (possibly non-compact) analytic manifold, thus has locally finitely (hence countably) many connected components, again by {\L}ojasiewicz stratification theorem. Since $f_{\bsa}$ is constant on each component of $\caC$, it follows that the set $T_{\bsa}$ is at most countable. We also remark that the set $\wt{\caC}$ is bounded by $\norm{\bsa}$; indeed, for example when $x>\norm{\bsa}$ we have that
	\beq\label{eq:convex_no_C}
	\frac{\partial f_{\bsa}}{\partial \re z}(x)
	=\frac{\partial f_{\bsa}}{\partial z}(x)+\frac{\partial f_{\bsa}}{\partial \ol{z}}(x)
	=2\Brkt{(\re \bsa-x) \frac{1}{\absv{(\bsa-x)^{2}}^{2}}}<0.
	\eeq
\end{rem}

\begin{rem}[Localized result]\label{rem:sharp_local}
	Notice that Theorem \ref{thm:sharp} (iii) and (iv) is a dichotomy of the density, around each point in the whole boundary $\partial\caD_{t}$ under the global assumption \eqref{eq:assump_main}. With only minor modification to the proof, we can prove a similar result locally and quantitatively for a given point on the boundary;
	if $z\in\partial\caD_{t}$ and $z\in\C\setminus\spec(\bsa)$, then either (iii) or (iv) holds true for $\rho$ around $z$ depending on whether $\absv{\nabla f_{\bsa}(z_{0}}$ is zero or not. In this case, for each fixed constant $C>0$, the implicit constants in \eqref{eq:edge_sharp} and \eqref{eq:edge_quad} can be chosen uniformly over $\bsa$ and $z$ satisfying $\norm{\bsa}\leq C$ and $\norm{(\bsa-z)^{-1}}\leq C$.
\end{rem}

\begin{rem}[Shape of $\supp\rho_{\bsa+\sqrt{t}\bsx}$ around a quadratic edge]
	For $z\in\caC_{t}$, the (local) shape of the $\caD_{t}$ is determined by the type of the critical point $z$. 
	If $z$ is a local minimum of $f_{\bsa}$, in particular if $\det H[f_{\bsa}](z)>0$, then the support is locally simply connected; see Figure~\ref{fig:44}.
	If $\det H[f_{\bsa}](z)<0$, then the support locally resembles a double cone; see Figure~\ref{fig:42}. The shape when $\det H[f_{\bsa}](z)=0$ is sensitive to higher partial derivatives of $f_{\bsa}$ in the degenerate direction; see Figure \ref{fig:tang}. We also remark that, after our work first appeared, the density of $\rho_{\bsa+\sqrt{t}\bsx}$ around such a degenerate critical point was studied in detail in the recent work \cite{Alt-Kruger2024Brown}.
\end{rem}

In the next theorem, we construct an $\bsa$ so that $\rho_{\bsa+\sqrt{t}\bsx}$ has any other rates of decay than quadratic at a specific point on the edge; see Appendix \ref{append:higher} for the proof.
\begin{thm}\label{thm:higher}
	Fix $p,q>0$ and let $\bsa$ be normal whose spectral distribution has a density on $\R^{2}\cong\C$ given by
	\beq\label{eq:two_side_Jacobi}
	\rho_{\bsa}(x+\ii y)={(p+1)(q+1)}x^{p}y^{q}\lone_{[0,1]}(x)\lone_{[0,1]}(y).
	\eeq
	If $t>0$ is such that
	\beq
	\frac{1}{t}>f_{\bsa}(0)={(p+1)(q+1)}\int_{0}^{1}\int_{0}^{1}\frac{x^{p}y^{q}}{x^{2}+y^{2}}\dd x\dd y,
	\eeq
	then the density $\rho$ of $\rho_{\bsa+\sqrt{t}\bsx}$ satisfies, for each constant $c>0$, 
	\beq
	\rho(z)\sim_{c} \absv{z}^{p+q}\qquad \text{as}\,\,z\to 0\,\,\text{in}\,\,\{z:0<c\re z\leq \im z\leq c^{-1}\re z\}.
	\eeq
\end{thm}
Note that if $\bsa$ is as in Theorem \ref{thm:higher} with $p,q\leq 0$, then $f_{\bsa}\equiv\infty$ on $\spec(\bsa)=[0,1]^{2}$, so that $\bsa$ satisfies Assumption \ref{assump} for all $t>0$. On the other hand, using that $\supp\rho_{\bsa}=[0,1]^{2}$ is convex, one can easily show that $\absv{\nabla f_{\bsa}}>0$ on $\C\setminus\spec(\bsa)$ following \eqref{eq:convex_no_C}. Hence, for $p,q\leq 0$, the density $\rho$ has sharp decay at the edge, i.e. $\rho\sim \lone_{\ol{\caD}_{t}}$ by Theorem \ref{thm:sharp} (ii) and (iii).

As a last result, we prove that for general $\bsa\in\caM$ the number of connected components of $\supp\rho_{\bsa+\sqrt{t}\bsx}$ decreases, modulo the spectrum of $\bsa$, as $t$ increases. See Appendix \ref{append:conn} for its proof.
\begin{thm}\label{thm:conn}
	The number of connected components of $\ol{\caD}_{t}\cup\spec(\bsa)$ decreases in $t>0$.
\end{thm}
Notice that the set $\ol{\caD}_{t}\cup\spec(\bsa)$ is closed; the connected components in Theorem \ref{thm:conn} are defined with respect to the relative topology of $\ol{\caD}_{t}\cup\spec(\bsa)$ in $\C$. Also, Theorem \ref{thm:conn} remains true even if the number of connected components is infinite at some time $t_{0}>0$. In this case, if the number becomes finite at some $t>t_{0}$, then it can never be infinite at any later time than $t$.

\begin{rem}\label{rem:R}
	The union with the spectrum $\spec(\bsa)$ in Theorem \ref{thm:conn} is included to bypass a fundamental difficulty. Namely, it is due to fact that the converse of the trivial inclusion $\supp\rho_{\bsa}\subset\spec(\bsa)$ is not true in general (see Remark \ref{rem:supp}). If $\bsa$ is such that $\supp\rho_{\bsa}=\spec(\bsa)$, then one can remove $\spec(\bsa)$ from the statement of Theorem \ref{thm:conn} simply using \eqref{eq:supp<D}. By \cite[Theorem 4.3]{Dykema-Noles-Zanin2018}, a notable class of such $\bsa$'s is \emph{decomposable} operators, which includes normal operators and finite square matrices. Most importantly $*$-limits of generic (non-Hermitian) random matrices (so-called \emph{DT-operators}) were proved to be decomposable in \cite{Dykema-Haagerup2004}.
	
	Certain $R$-diagonal operators $\bsa$ also serve as examples and counterexamples for $\supp\rho_{\bsa}=\spec(\bsa)$ -- 
	recall that an $R$-diagonal operator is the product $\bsu\bsh$ of a $*$-free pair $(\bsu,\bsh)$ with Haar unitary $\bsu$ and positive $\bsh$. 
	In \cite[Proposition 4.6]{Haagerup-Larsen2000} it is proved that $\supp\rho_{\bsu\bsh}=\spec(\bsu\bsh)$ if and only if $\norm{\bsh^{-1}}<\infty$ or $\brkt{\bsh^{-2}}=\infty$. 
	More precisely, $\supp\rho_{\bsu\bsh}$ is always the annulus with inner and outer radii $(\brkt{\bsh^{-2}}^{-1/2},\brkt{\bsh^{2}}^{1/2})$, 
	and if $\norm{\bsh^{-1}}=\infty$ and $\brkt{\bsh^{-2}}<\infty$, then $\spec(\bsu\bsh)$ is the ball with radius $\brkt{\bsh^{2}}^{1/2}$ so that $\supp\rho_{\bsu\bsh}\subsetneq\spec(\bsu\bsh)$.
\end{rem}

\subsection{Checkability of Assumption \ref{assump}}\label{sec:assump_check}
In this section, we discuss how to practically check Assumption \ref{assump} for a given $\bsa$. First of all, notice that Assumption \ref{assump} can be written solely in terms of $\{\rho_{\absv{\bsa-z}}:z\in\C\}$. Namely, it is equivalent to
\beq\label{eq:assump}
f_{\bsa}(z)=\int_{\R}\frac{1}{x^{2}}\dd\rho_{\absv{\bsa-z}}(x)>\frac{1}{t}\quad\text{whenever}\quad 0\in\supp\rho_{\absv{\bsa-z}}.
\eeq

When $\bsa$ is normal, $\rho_{\absv{\bsa-z}}$ is the push-forward of $\rho_{\bsa}$ by the map $\absv{\cdot-z}$. Thus \eqref{eq:assump} becomes a statement depending only on the measure $\rho_{\bsa}$ and hence can be validated with standard measure theoretic techniques. For example,  if $\bsa$ is normal and $\rho_{\bsa}$ has a strictly positive density in domain a with a regular (Lipschitz) boundary, then $f_{\bsa}(z)=\infty$ for all $z\in\supp\rho_{\bsa}=\spec\bsa$.
As a side note,  we remark that for any Borel measure $\rho$ on $\C$ we have (see \cite[Lemma 4.4]{Zhong2021} for a proof)
\beq\label{eq:assump_check}
\int_{\C}\frac{1}{\absv{z-w}^{2}}\dd\rho(w)=\infty\qquad \text{for $\rho$-a.e. }z\in\C.
\eeq

On the other hand when $\bsa$ is a general non-normal operator, we can not treat $(\bsa-z)^{-1}$ as an integral over the spectral resolution of $\bsa$, and the relation between $\rho_{\bsa}$ and $\rho_{\absv{\bsa-z}}$ is unclear.   Even in this case, we still have the inequality 
\beq\label{eq:Lp}
f_{\bsa}(z)=\int\frac{1}{x^{2}}\dd\rho_{\absv{\bsa-z}}(x)\geq \int_{\C}\frac{1}{\absv{w-z}^{2}}\dd\rho_{\bsa}(w)\qquad \forall z\in\C,
\eeq
from \cite[Proposition 2.14 and Theorem 2.19]{Haagerup-Schultz2007}. The inequality \eqref{eq:Lp} can be used to check $f_{\bsa}(z)>1/t$ for a given point $z$; e.g. when $\bsa$ itself is a circular element, then \eqref{eq:Lp} immediately implies $f_{\bsa}(z)=\infty$ for all $z\in\ol{\D}=\rho_{\bsa}=\spec(\bsa)$, so that Assumption \ref{assump} holds for all $t>0$.
However,  the inequality in \eqref{eq:Lp} is strict in general. An $R$-diagonal operator $\bsu\bsh$ in Remark \ref{rem:R} is such an example; we have
\beq
f_{\bsu\bsh}(0)=\brkt{\bsh^{-2}}=\frac{1}{\dist(0,\supp\rho_{\bsu\bsh})^{2}}\geq \int_{\C}\frac{1}{\absv{z}^{2}}\dd\rho_{\bsu\bsh}(z),
\eeq
and the last inequality is strict unless $\rho_{\bsh}$ is a point mass due to \cite[Proposition 4.6]{Haagerup-Larsen2000}. Also recall from Remark \ref{rem:R} that $\supp\rho_{\bsu\bsh}\subsetneq\spec(\bsu\bsh)$ when $\brkt{\bsh^{-2}}<\infty$ and $\norm{\bsh^{-1}}=\infty$.


For certain operators $\bsa$, there is an alternative approach to the $z$-dependent measure $\rho_{\absv{\bsa-z}}$. More precisely, this is the case when the Hermitized Green function $\bsM_{\bsa}$ (see \eqref{eq:M_def} later for its definition) solves a reasonable Schwinger-Dyson equation. Typical examples include $R$-diagonal operators, for which we have the identity (see \cite[Proposition 3.1]{Haagerup-Kemp-Speicher2010} or \cite[Proposition 3.5]{Haagerup-Larsen2000})
\beq\label{eq:R_L2}
\wt{\rho}_{\absv{\bsu\bsh-z}}=\wt{\rho}_{\bsh}\boxplus\wt{\delta}_{\absv{z}}
\eeq
where $\wt{\mu}$ denotes the symmetrization of a measure $\mu$ on $\R_{+}$ to $\R$. 

Thus, via the identity \eqref{eq:R_L2}, we can check \eqref{eq:assump} in terms of $\wt{\rho}_{\bsh}$; for example when $\bsa$ is a circular element it is known that $\wt{\rho}_{\bsh}$ is the semi-circle distribution, for which we can easily check that the following three conditions on $z$ are all equivalent:
\begin{equation}
	0\in\supp \wt{\rho}_{\bsh}\boxplus\wt{\delta}_{\absv{z}},\qquad \absv{z}\leq 1,\qquad  \int_{\R}\frac{1}{x^{2}}\dd(\wt{\rho}_{\bsh}\boxplus\wt{\delta}_{\absv{z}})(x) =f_{\bsa}(z)=\infty.
\end{equation}
Another example is when $\bsa=(\sqrt{s_{ij}}\bsx_{ij})_{ij}\in(\caM)^{n\times n}$ is a matrix-valued circular element, where $n$ is fixed, $\bsx_{ij}$'s are $*$-free circular elements, and $s_{ij}>0$. In this case, $M_{\bsa}$ solves a Schwinger-Dyson equation similar to that of a usual circular element, and by studying the equation it was proved in \cite[Proposition 3.2]{Alt-Erdos-Kruger2018} that the following three conditions on $z$ are equivalent:
\begin{equation}
	0\in\supp\rho_{\absv{\bsa-z}},\qquad\absv{z}\leq \rho(S),\qquad
	\brkt{\absv{\bsa-z}^{-2}}=\infty,
\end{equation}
where $\rho(S)$ is the spectral radius of the entrywise positive matrix $S=(s_{ij})\in\C^{n\times n}$. In particular, for this choice of $\bsa$ Assumption \ref{assump} holds for all $t>0$.

\section{Examples}\label{sec:ex}

In this section, we present five examples of $\rho_{\bsa+\sqrt{t}\bsx}$ for various $\bsa$'s, along with numerical simulations. In Examples \ref{ex:A4+X} -- \ref{ex:tang} we consider $\bsa$ satisfying Assumption \ref{assump}, hence they serve as illustrations of Theorem~\ref{thm:sharp}. Example \ref{ex:higher} provides an illustration of Theorem \ref{thm:higher}, for which the density of $\rho_{\bsa+\sqrt{t}\bsx}$ has faster than quadratic decay. Finally in Example \ref{ex:Jacobi}, we consider $\bsa$'s violating Assumption~\ref{assump}, for which the density of $\rho_{\bsa+\sqrt{t}\bsx}$ still has sharp decay at the edge but the shape of $\supp\rho_{\bsa+\sqrt{t}\bsx}$ is different from those described in Theorem \ref{thm:sharp}.

In order to visualize the density of $\rho_{\bsa+\sqrt{t}\bsx}$, we use the eigenvalues of the corresponding random matrix model. More precisely, by \cite[Theorem 6]{Sniady2002}, if the $*$-distribution of an $(N\times N)$ matrix $A$ converges to that of an operator $\bsa$ as $N\to\infty$ and $X$ is an $(N\times N)$ complex Ginibre matrix, then the eigenvalue distribution of $A+\sqrt{t}X$ converges to the Brown measure $\rho_{\bsa+\sqrt{t}\bsx}$. In this regard, all of Figures \ref{fig:4}--\ref{fig:Jacobi} contain sampled eigenvalues of an $(N\times N)$ matrix $A+\sqrt{t}X$ with $N\sim 10^{3}$, where $X$ is a complex Ginibre matrix and the $*$-distribution of $A$ is close to that of $\bsa$ (which varies by examples).

\begin{ex}\label{ex:A4+X}
	Let $\bsa$ be a normal operator with spectral measure $(\delta_{1}+\delta_{-1}+\delta_{\ii}+\delta_{-\ii})/4$. Then the following can be proved for the Brown measure $\rho_{\bsa+\sqrt{t}\bsx}$ with direct, elementary calculations using Theorems \ref{thm:prev} and \ref{thm:sharp} (see Figure~\ref{fig:4} for an illustration):
	\begin{itemize}
		\item[(a)] For $t<2(\sqrt{2}-1)$, the support $\ol{\caD}_{t}$ consists of four simply connected regions, and $\caC_{t}=\emptyset$ so that the density $\rho$ of $\rho_{\bsa+\sqrt{t}\bsx}$ is uniformly bounded from below on $\caD_{t}$ by Theorem \ref{thm:sharp} (i) and (ii).
		
		\item[(b)] At $t=2(\sqrt{2}-1)$, the support $\ol{\caD}_{t}$ becomes connected, but there are four critical points in $\caC_{t}$ and we have $\det H[f_{\bsa}](z)<0$ for each $z\in\caC_{t}$. Thus, around each $z\in\caC_{t}$, the density has quadratic decay as described in \eqref{eq:edge_quad_1} with $\caS_{z,\kappa}=\C$, and has a (locally) hourglass-shaped support. The density is bounded from below away from $\caC_{t}$.
		
		\item[(c)] For $2(\sqrt{2}-1)<t<1$, the support $\ol{\caD}_{t}$ is homeomorphic to an annulus, and $\caC_{t}=\emptyset$ so that the density is bounded from below on $\caD_{t}$. 
		
		\item[(d)] At $t=1$, the support $\ol{\caD_{t}}$ becomes simply connected. In this case $\caC_{t}=\{0\}$ and $\det H[f_{\bsa}](0)>0$, so that $\caS_{0,\kappa}=\C$ and the density has quadratic decay around the origin in every direction.
		
		\item[(e)] For $t>1$, the density is strictly positive and supported in a single simply connected region.
	\end{itemize}
\end{ex}
\begin{figure}
	\centering
	\subcaptionbox{t=0.5}{\includegraphics[width=0.3\textwidth]{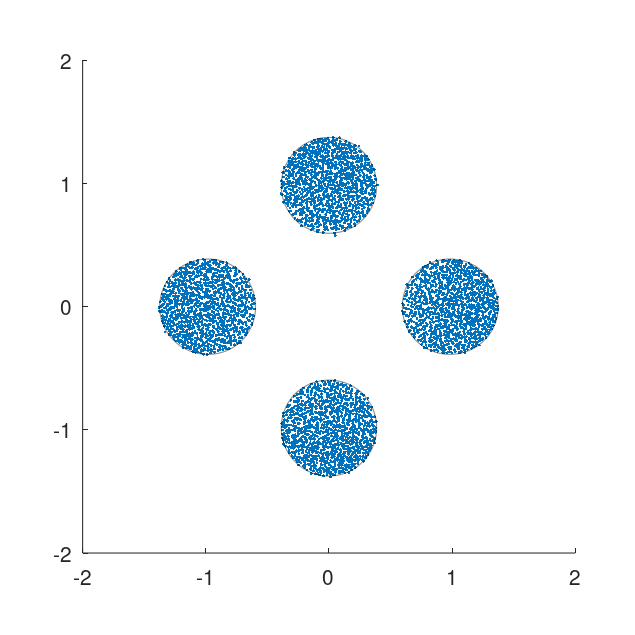}}
	\subcaptionbox{t=$2(\sqrt{2}-1)$\label{fig:42}}{\includegraphics[width=0.3\textwidth]{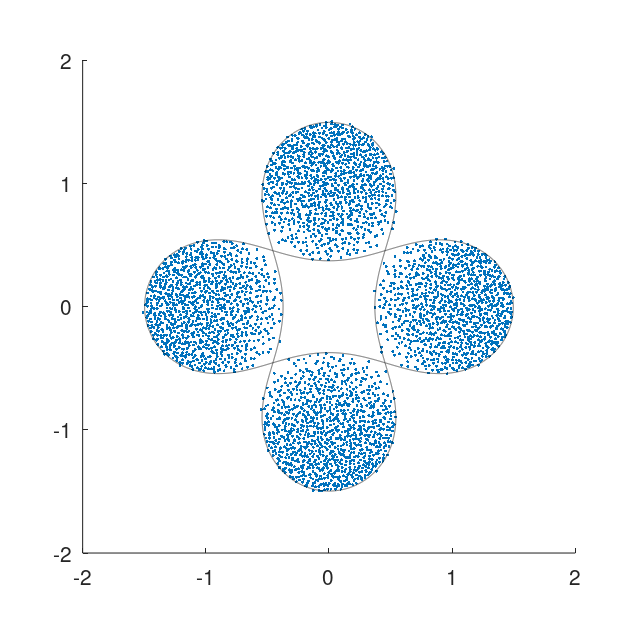}}
	\subcaptionbox{t=0.9}{\includegraphics[width=0.3\textwidth]{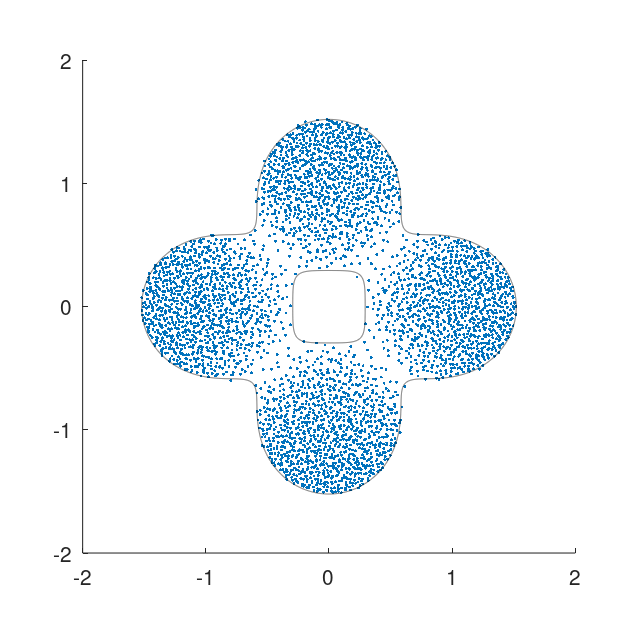}}
	\subcaptionbox{t=1\label{fig:44}}{\includegraphics[width=0.3\textwidth]{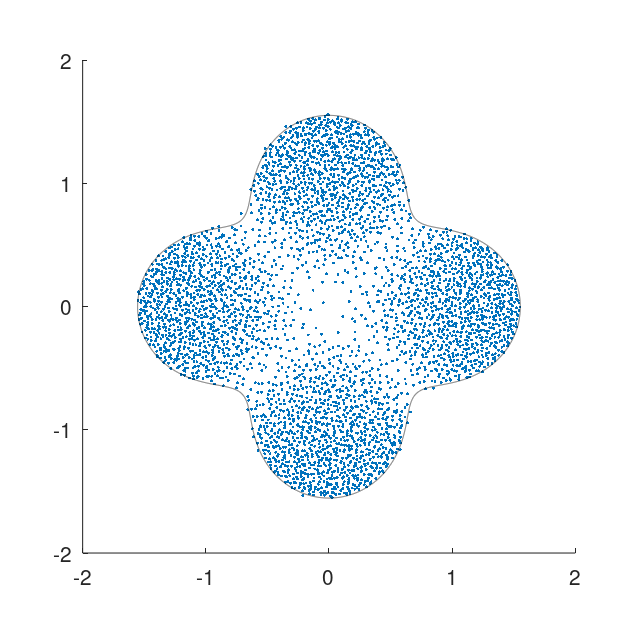}}
	\subcaptionbox{t=1.1}{\includegraphics[width=0.3\textwidth]{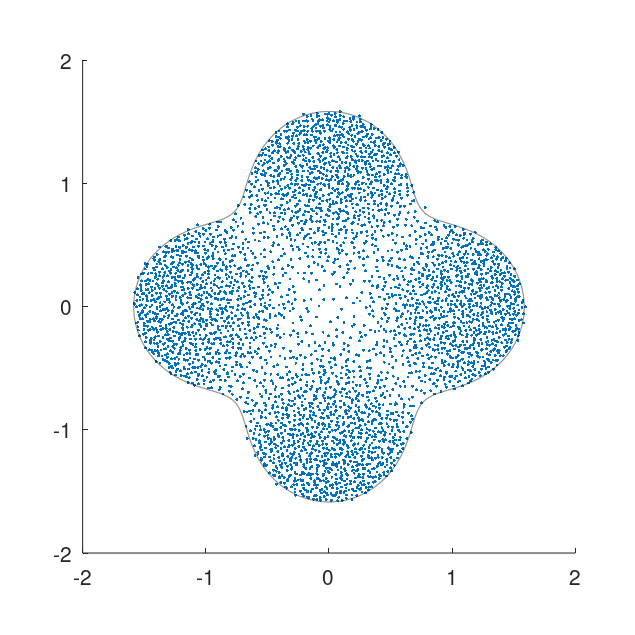}}
	\caption{Sampled eigenvalues of $A+\sqrt{t}X$ (blue dots) and the boundary of $\supp\rho_{\bsa+\sqrt{t}\bsx}$ (black curve), where $A$ is a diagonal matrix with entries $\pm1$ or $\pm\ii$ with the same multiplicity}\label{fig:4}
\end{figure}

\begin{ex}\label{ex:U+X}
	A similar phenomenon as in Example \ref{ex:A4+X} remains true when $\bsa$ is normal with spectral measure uniformly distributed on $k$-th roots of unity, as long as $k>2$. This even carries over along the limit $k\to\infty$, leading to $\bsu+\sqrt{t}\bsx$ with Haar unitary $\bsu$. We can easily see that the measure $\rho_{\bsu+\sqrt{t}\bsx}$ is supported on an annulus or a disk (depending on $t$) and radially symmetric. The set $\caC_{t}$ is empty unless $t=1$, in which case $\caC_{t}=\{0\}$ and the density has quadratic decay at the origin. See Figure \ref{fig:u} for an illustration; we sampled the eigenvalues of $J+X$, where $J=(J_{ij})_{1\leq i,j\leq N}$ is the $N\times N$ null Jordan block, i.e. $J_{ij}=\lone(j-i=1)$. Since the $*$-distribution of $J$ converges to that of $\bsu$ as $N\to\infty$, we can see that the eigenvalue distribution of $J+X$ converges to the Brown measure $\rho_{\bsu+\bsx}$. In fact, $J$ is a rank-one perturbation of $J'_{ij}\deq \lone(j-i\equiv 1 \mod N)$ which is a unitary matrix with uniform spectral measure on the set of $N$-th roots of unity.
\end{ex}

\begin{figure}
	\centering
	\subcaptionbox{$\spec(J+X)$ and $\supp\rho_{\bsu+\bsx}$\label{fig:u}}{\includegraphics[width=0.3\textwidth]{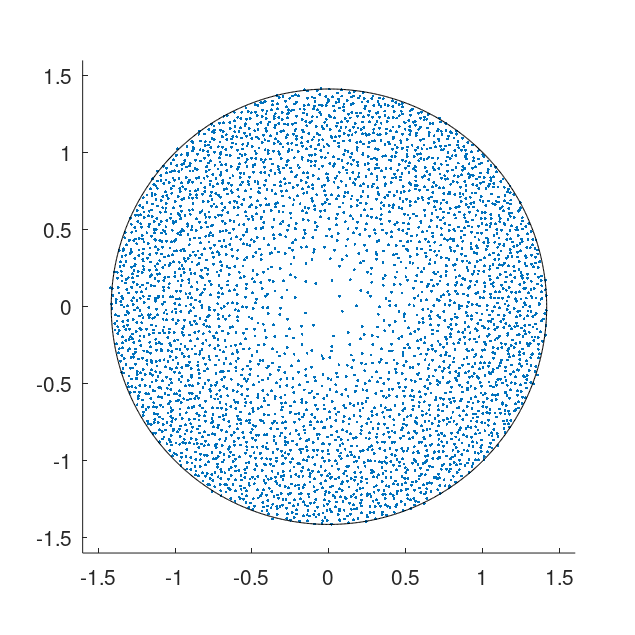}}
	\hspace{4em}
	\subcaptionbox{$\caS(0,\kappa)$, $\spec(B+\sqrt{t_{0}}X)$, and $\supp\rho_{\bsb+\sqrt{t_{0}}\bsx}$ \label{fig:tang}}{\includegraphics[width=0.3\textwidth]{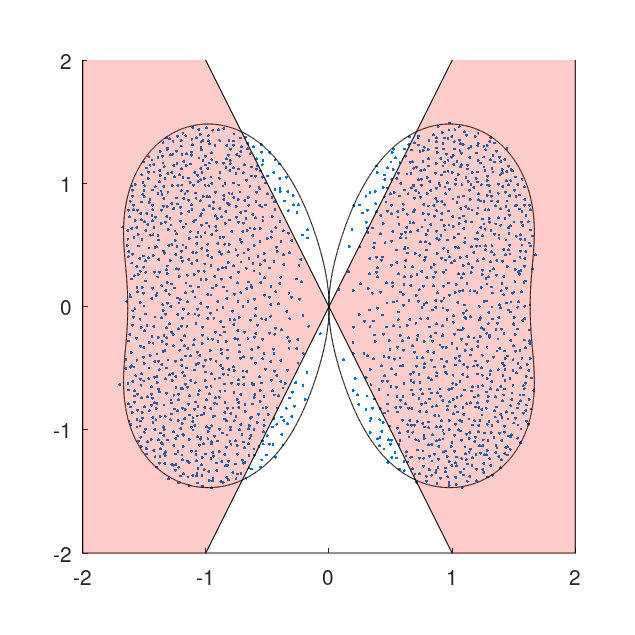}}
	\caption{$\supp\rho_{\bsa+\sqrt{t}\bsx}$ (black) with different choices of $\bsa$ and eigenvalues (blue) of the corresponding matrix models, see Examples \ref{ex:U+X}--\ref{ex:tang} for details  \label{fig:abc}}
	
\end{figure}

\begin{ex}\label{ex:tang}
	In Examples \ref{ex:A4+X} and \ref{ex:U+X}, every $z\in\caC_{t}$ was a non-degenerate critical point, i.e. $\det H[f_{\bsa}](z)\neq 0$. This implies $P_{z}=0$ and thus $\caS_{z,\kappa}=\C$ ($P_{z}$ and $\caS_{z,\kappa}$ are defined in Theorem \ref{thm:sharp} (iv)). In this example, we construct an operator $\bsb$ so that $\det H[f_{\bsb}](z)=0$ for some $z\in\caC_{t}$. Consider a normal $\bsb$ whose spectral measure is given by
	\beqs
	\int_{\C}\varphi(z)\dd\rho_{\bsb}(z)= \frac{35}{96}\int_{-1}^{1}\frac{\varphi(-1+\ii y)+\varphi(1+\ii y)}{2}(1+y^{2})^{3}\dd y,\qquad \varphi\in C_{0}(\C).
	\eeqs
	More concretely, we may take $\bsb$ to be the complex random variable $x+\ii y$ where $x$ and $y$ are independent, $x$ is $\pm1$ Bernoulli distributed, and $y$ is a continuous random variable with distribution function given by
	\beq
	\P[y\leq t]=\frac{35}{96}\int_{-1}^{t}(1+s^{2})^{3}\dd s,\qquad t\in[-1,1].
	\eeq
	After some direct computations, we find that
	\beq
	\nabla f_{\bsb}(0)=0,\qquad \frac{\partial^{2}}{\partial x^{2}}f_{\bsb}(0)>0,\qquad \frac{\partial^{2}}{\partial x\partial y}f_{\bsb}(0)=0,\qquad \frac{\partial^{2}}{\partial y^{2}}f_{\bsb}(0)=0,
	\eeq
	so that $\det H[f_{\bsb}](0)=0$ and $P_{z}$ is the projection onto the $y$-coordinate. Taking $t_{0}=1/f_{\bsb}(0)$, it follows that the origin is a critical boundary point for $\rho_{\bsb+\sqrt{t_{0}}\bsx}$. As $t$ increases to $t_{0}$, the two disjoint regions of $\ol{\caD}_{t}$ become tangent to each other at the origin. Even in this case, by Theorem \ref{thm:sharp} (iv), the density around the origin is proportional to $\absv{z}^{2}$ in any angular sector $\caS(0,\kappa)$ contained in the open set $\{z:\absv{\re z}>0\}$. Figure \ref{fig:tang} shows the eigenvalues of $B+\sqrt{t_{0}}X$ (blue), $\supp\rho_{\bsb+\sqrt{t_{0}}\bsx}$ (black), and $\caS(0,\kappa)$ (red), where $N=10^{3}$ and $B=\diag(b_{i})$ is the $(2N\times 2N)$ diagonal matrix given by
	\beqs
	b_{i}=\begin{cases}
		-1+\ii \wt{y}_{i}, & 1\leq i\leq N \\
		1+\ii\wt{y}_{i-N}, & N+1\leq i\leq 2N,
	\end{cases}
	\qquad
	\frac{35}{96}\int_{-1}^{\wt{y}_{i}}(1+s^{2})^{3}\dd s=\frac{i}{N}. 
	\eeqs
	 Note that in Figure \ref{fig:tang} we see two curves since $\partial_{y}^{3}f_{\bsa}(0)=0$ and $\partial_{y}^{4}f_{\bsa}(0)<0$; if $\partial_{y}^{3}f_{\bsa}(0)$ did not vanish, we would see one curve with a cusp singularity.
\end{ex}

\begin{ex}\label{ex:higher}
	In this example we show a numerical illustration for Theorem \ref{thm:higher}. Consider the random diagonal (hence normal) matrix $C$ of size $10^{3}$, whose entries are i.i.d. samples from the law $\rho_{\bsa}$ in \eqref{eq:two_side_Jacobi} with $p=1$ and $q=1.5$. In this case, the critical value $t_{0}=f_{\bsa}(0)^{-1}$ in Theorem \ref{thm:higher} is $t_{0}\approx 0.32$. In Figure \ref{fig:higher} we plot the eigenvalues of $10$ samples of $C+\sqrt{t}X$ for two values of $t$. When $t<t_{0}$, we see from Figure \ref{fig:higher1} that the density of $\rho_{\bsa+\sqrt{t}\bsx}$ decays around $0$. In contrast when $t>t_{0}$, Figure \ref{fig:higher2} shows that we recover the sharp cutoff edge behavior as in Theorem \ref{thm:sharp} (iii).
\end{ex}

\begin{figure}
	\centering
	\subcaptionbox{$t=0.2<t_{0}$\label{fig:higher1}}{\includegraphics[width=0.3\textwidth]{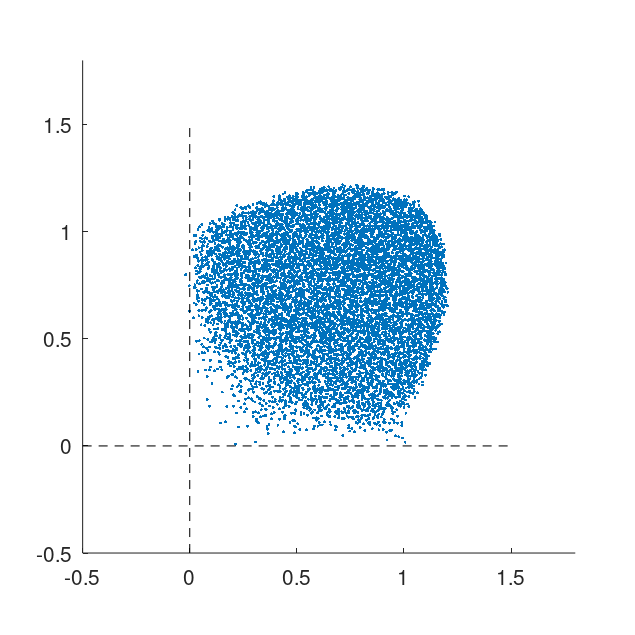}}
	\hspace{4em}
	\subcaptionbox{$t=0.9>t_{0}$\label{fig:higher2}}{\includegraphics[width=0.3\textwidth]{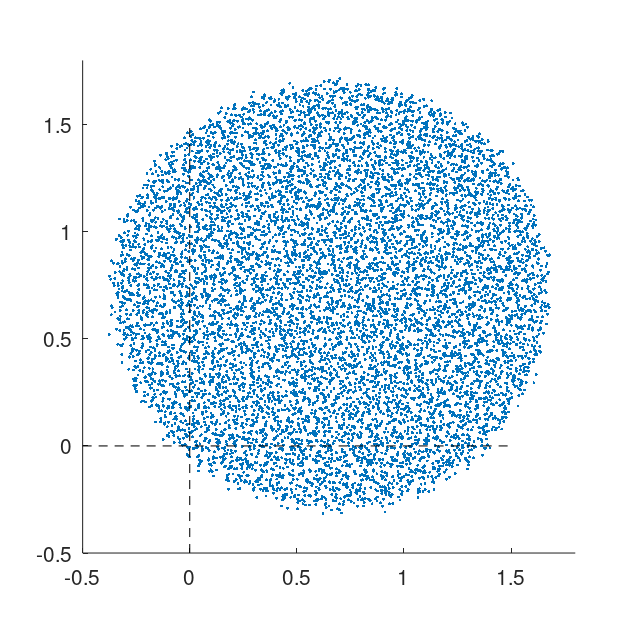}}
	\caption{$\spec(C+\sqrt{t}X)$ (blue) for two values of $t$, below and above the critical value, with $xy$-axes (dashed) for reference; see Example \ref{ex:higher} for details \label{fig:higher}}
\end{figure}
As a last example, we consider the Hermitian $\bsd$ whose spectral measure is of Jacobi-type, i.e. has a density (also denoted by $\rho_{\bsd}$) satisfying for some $p_{\pm}>-1$
\beq\label{eq:Jacobi}
\rho_{\bsd}(x)\sim x^{p_{-}}(1-x)^{p_{+}}\lone_{[0,1]}(x).
\eeq
In other words, $\rho_{\bsd}$ has power-law decays around lower and upper edges, $0$ and $1$, with exponents $p_{-}$ and $p_{+}$, respectively. We first recall the following result for the Brown measure of $\bsd+\sqrt{t}\bsx$ with general Hermitian $\bsd$ from \cite[Theorem 1.1]{Ho-Zhong2023}.
\begin{itemize}
	\item[(i)] There is a continuous function $v_{t}:\R\to[0,\infty)$ such that $\caD_{t}=\{x+\ii y\in\C:\absv{y}<v_{t}(x)\}$. More precisely, $v_{t}(x)=v_{t}(x,0)$ is defined in \eqref{eq:def_v}.
	\item[(ii)] The density of $\rho_{\bsd+\sqrt{t}\bsx}$ is constant along vertical direction.
\end{itemize}
If we specialize $\bsd$ to have Jacobi-type spectral measure, we have the following result; see Appendix \ref{append:Jacobi} for a proof.
\begin{prop}\label{prop:Jacobi}
	Let $\bsd\in\caM$ be $*$-free of a circular element $\bsx$ such that $\rho_{\bsd}$ satisfies \eqref{eq:Jacobi}. Define four real numbers.
	\beq\begin{aligned}
		t_{-}\deq&\frac{1}{f_{\bsd}(0)},&\qquad 
		x_{t,-}\deq&\sup\left\{x<0:f_{\bsd}(x)\leq \frac{1}{t}\right\}, 	\\
		t_{+}\deq&\frac{1}{f_{\bsd}(1)},&\qquad 
		x_{t,+}\deq&\inf\left\{x>1:f_{\bsd}(x)\leq \frac{1}{t}\right\},
	\end{aligned}\eeq
	where $t_{\pm}$ can be zero. Then the following hold true.
	\begin{itemize}
		\item[(iii)] $v_{t}\vert_{\R}$ is strictly positive on $(x_{t,-},x_{t,+})$ and identically zero elsewhere.
		\item[(iv)] If $t\in(0,t_{-})$, then $x_{t,-}=0$ and $v_{t}(x)\sim x^{p_{-}}\lone(x>0)$ as $x\to0.$ Note that $t_{-}>0$ only if $p_{-}>1$.
		\item[(v)] If $t\in(t_{-},\infty)$, then $x_{t,-}<0$ and $v_{t}(x)\sim \sqrt{(x-x_{t,-})_{+}}$ as $x\to x_{t,-}$.
		\item[(vi)] If $t\neq t_{\pm}$, the density of $\rho_{\bsd+\sqrt{t}\bsx}$ is bounded from below by a positive constant in $\caD_{t}$.
	\end{itemize}
	The same statements as in (iv) and (v) hold true at the upper edge $x_{t,+}$ after obvious modifications.
\end{prop}
For $p_{-}>1$ and $t<t_{-}$, Assumption \ref{assump} fails at the lower edge $x_{t,-}=0$ and $\caD_{t}$ has a cusp of order $p_{-}$ at $x_{t,-}$. Note that under Assumption \ref{assump} such a sharp cusp would not happen. To see this, note that $\partial\caD_{t}$ is a level set of the smooth function $f_{\bsd}$ with $\Delta f_{\bsd}>0$, hence the open set $\caD_{t}$ contains a non-trivial angular sector based at each $z\in\partial\caD_{t}$. Still, the density is comparable to $\lone_{\ol{\caD}_{t}}$ in all cases; we expect Proposition \ref{prop:Jacobi} (vi) to be true even for $t=t_{\pm}$, but we excluded these cases for technical reasons.

%

\begin{ex}\label{ex:Jacobi}
	Finally, we present an illustration for Proposition \ref{prop:Jacobi}. Here we consider the random diagonal matrix $D$ whose entries are sampled from $\mathrm{Beta}(3,4)$, i.e. with density proportional to \eqref{eq:Jacobi} with $p_{-}=2$ and $p_{+}=3$. In this case, the limiting eigenvalue distribution of $D+\sqrt{t}X$ is $\rho_{\bsd+\sqrt{t}\bsx}$ where $\bsd$ is Hermitian with $\bsd\sim\mathrm{Beta}(3,4)$, and trivial computations show that the two critical values for $t$ in Proposition \ref{prop:Jacobi} are $t_{-}=1/15$ and $t_{+}=1/5$. In Figure \ref{fig:Jacobi}, we plot $\supp\rho_{\bsd+\sqrt{t}\bsx}$ and eigenvalues of $D+\sqrt{t}X$ for three values of $t$ divided by $t_{-}$ and $t_{+}$. For $t<t_{-}$ (Figure~\ref{fig:Jacobi1}), the support $\ol{\caD_{t}}$ has cusps at both left and right edges, respectively with orders $p_{-}=2$ and $p_{+}=3$. When $t\in(t_{-},t_{+})$, (Figure \ref{fig:Jacobi2}) only the right edge remains as a cusp and the left edge becomes smooth. Finally for $t>t_{+}$ (Figure \ref{fig:Jacobi3}) the boundary $\partial\caD_{t}$ becomes smooth. In all three cases the density remains bounded from below up to the boundary, in contrast to Figure \ref{fig:higher1} where it vanishes at $0$.
\end{ex}

\begin{figure}
	\centering
	\subcaptionbox{$t=1/30<t_{-}$\label{fig:Jacobi1}}{\includegraphics[width=0.3\textwidth]{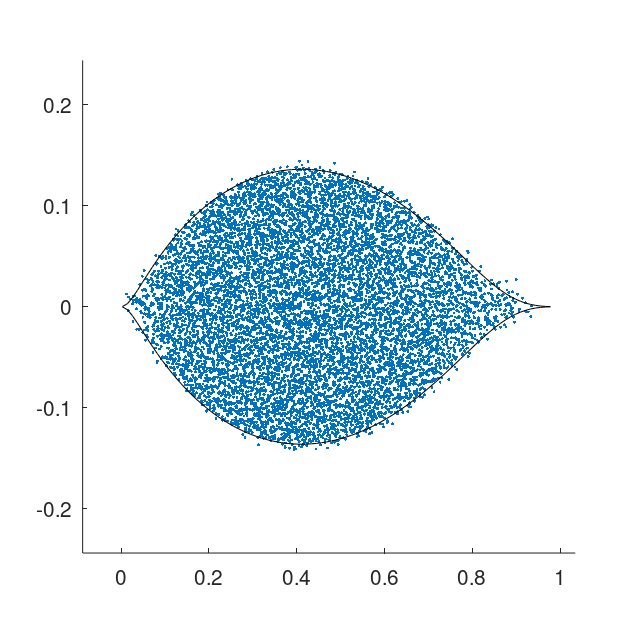}}
	\subcaptionbox{$t_{-}<t=1/10<t_{+}$\label{fig:Jacobi2}}{\includegraphics[width=0.3\textwidth]{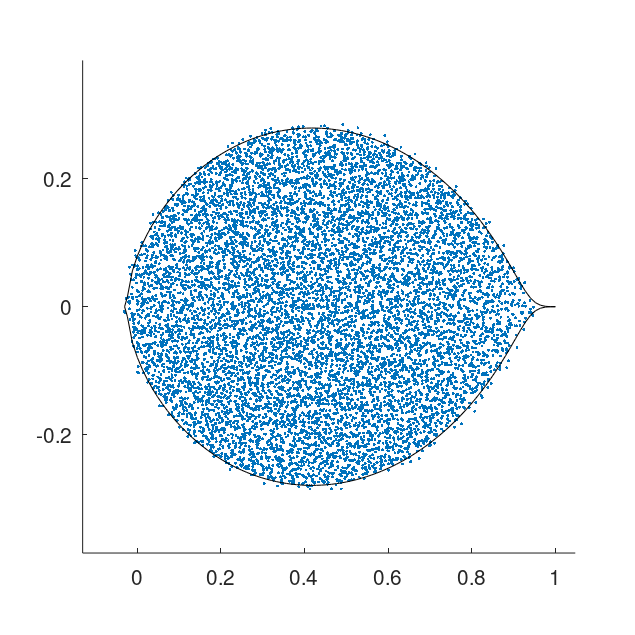}}
	\subcaptionbox{$t=2/5>t_{+}$\label{fig:Jacobi3}}{\includegraphics[width=0.3\textwidth]{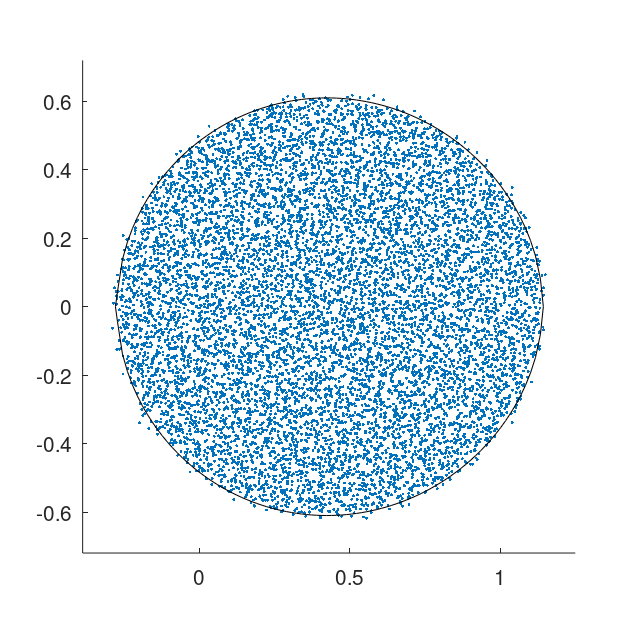}}
	\caption{$\spec(D+\sqrt{t}X)$ (blue) and $\supp\rho_{\bsd+\sqrt{t}\bsx}$ (black) for three values of $t$ (with $y$-axis scaled differently for clarity); see Example \ref{ex:Jacobi} for details\label{fig:Jacobi}.}
\end{figure}

\section{Preliminaries}\label{sec:prelim}
 In this section we collect various background information on $\bsa+\sqrt{t}\bsx$  that will be used later. They are largely taken from a series of papers \cite{Belinschi-Yin-Zhong2022,Bercovici-Zhong2022,Zhong2021} and we include them here
in order to put our proof in a proper context. We stress that all content in this section applies	to general $\bsa$ without Assumption 7. The new results, including the proof of Theorem~\ref{thm:sharp}, appear in Section 5. 

For simplicity, we use the shorthand notation
\beq\label{eq:def_at}
\bsa_{t}\deq \bsa+\sqrt{t}\bsx,\qquad t\geq 0,
\eeq 
in the rest of the paper, where the operators $\bsa,\bsx\in\caM$ are defined in Definition \ref{defn:model}.

\subsection{Hermitization}
As the definition \eqref{eq:Brown} of Brown measure suggests, we are naturally interested in the spectral measure of $\absv{\bsa_{t}-z}$ for varying $z\in\C$. We can further linearize the product $\absv{\bsa_{t}-z}^{2}=(\bsa_{t}-z)(\bsa_{t}-z)\adj$, that is, using the fact that $\absv{\bsa_{t}-z}$ has the same spectrum as
\beq\label{eq:linearize}
\begin{pmatrix}0 &\bsa_{t}-z\\(\bsa_{t}-z)\adj & 0 \end{pmatrix}\in M_{2}(\caM)=M_{2}(\C)\otimes\caM,
\eeq 
modulo symmetry about zero. The main advantage of studying \eqref{eq:linearize} instead of $\absv{\bsa_{t}-z}$ is that the linearized matrix is a plain sum of linearized versions of $\sqrt{t}\bsx$ and $\bsa-z$, which allows us to exploit their freeness more easily. The matrix in \eqref{eq:linearize} is often called the Hermitization of $\bsa_{t}$, a concept first introduced in \cite{Girko1984} along the proof of circular law. Hermitization method also has been proved to be an effective tool in free probability; we refer to \cite{Belinschi-Sniday-Speicher2018} for its application to  the study of Brown measures of non-normal operators. 

The goal of this section is to introduce notations involving Hermitization and its Green function. Define for each $\bsb\in\caM$
\beq\label{eq:Hermit_def}
\bsH_{\bsb}\deq \begin{pmatrix} 0 & \bsb \\ \bsb\adj & 0\end{pmatrix}\in M_{2}(\caM)\equiv M_{2}(\C)\otimes \caM.
\eeq
Obviously $H_{\bsb}$ is a bounded Hermitian operator, so that we may define its generalized resolvent as follows;
\beq\label{eq:UHP}
\bsG_{\bsb}(Z)\deq(\bsH_{\bsb}-Z)^{-1},\qquad Z\in\bbH_{+}(M_{2}(\C)),
\eeq
where we canonically embedded $\bbH_{+}(M_{2}(\C))$ in $M_{2}(\caM)$. As immediate consequences of the definition of $\bsG_{\bsb}(Z)$, we have 
\beq\label{eq:G}\begin{aligned}
	\im \bsG_{\bsb}(Z)&=\frac{1}{\bsH_{\bsb}-Z\adj}\im Z\frac{1}{\bsH_{\bsb}-Z}>0, \\
	\norm{\bsG_{\bsb}(Z)}&=\left(\inf_{\norm{\bsv}=1}\absv{\bsv\adj(\bsH_{\bsb}-Z)\bsv}\right)^{-1}\leq \norm{(\im Z)^{-1}},
\end{aligned}\eeq
where the infimum is taken over vectors $\bsv$ in the GNS Hilbert space $L^{2}(\caM,\brkt{\cdot})$ on which $M_{2}(\caM)$ acts.
With a slight abuse of notation we use $\brkt{\cdot}$ to denote the (block-wise) partial trace on $M_{2}(\caM)$, that is,
\beq
\Brkt{\begin{pmatrix} \bsb_{11} & \bsb_{12} \\ \bsb_{21} & \bsb_{22}\end{pmatrix}}\deq \begin{pmatrix}\brkt{\bsb_{11}} & \brkt{\bsb_{12}} \\ \brkt{\bsb_{21}} & \brkt{\bsb_{22}}\end{pmatrix}\in M_{2}(\C).
\eeq
We then use the partial trace to define the $M_{2}(\C)$-valued Stieltjes transform as
\beq\label{eq:M_def}
\bsM_{\bsb}(Z)\deq \brkt{\bsG_{\bsb}(Z)},\qquad Z\in\bbH_{+}(M_{2}(\C)).
\eeq 
Note that the same trivial estimates as in \eqref{eq:G} hold for $\bsM$ and its normalized trace $\tr \bsM\deq \Tr \bsM/2$;
\beq\label{eq:M}\begin{aligned}
	\im \bsM_{\bsb}(Z)&> 0, & \im \tr \bsM_{\bsb}(Z)&> 0,	\\
	\norm{\bsM_{\bsb}(Z)}&\leq \norm{(\im Z)^{-1}}, \qquad& \absv{\tr \bsM_{\bsb}(Z)}&\leq \norm{(\im Z)^{-1}}.
\end{aligned}\eeq
In what follows, we often take the variable $Z\in \bbH_{+}(M_{2}(\C))$ of the form
\begin{equation}
	\Theta(z,\eta)\deq\begin{pmatrix} \ii\eta & z \\ \ol{z} & \ii\eta\end{pmatrix},\qquad z\in\C,\eta>0.
\end{equation}
 By Schur complement formula, we can easily check that $ M_{\bsb}(\Theta(z,\eta))$ is itself of the form $\Theta(\wt{z},\wt{\eta})$ for some $\wt{z}\in\C,\wt{\eta}>0$. In particular, $\tr M_{\bsb}(\Theta(z,\eta))$ is purely imaginary:
\begin{equation}\label{eq:Mim}
	\tr{M_{\bsb}(\Theta(z,\eta))}=\Brkt{\frac{\ii\eta}{\absv{\bsb-z}^{2}+\eta^{2}}}\in\ii (0,\infty),\qquad z\in\C,\eta>0.
\end{equation}


\begin{rem}\label{rem:full}
	As easily seen, the $M_{2}(\C)$-valued Stieltjes transform of $\bsb$ only determines $*$-moments of the form $\brkt{\absv{\bsb}^{2k}}$ or $\brkt{\absv{\bsb}^{2k}\bsb}$. Instead, a further generalization determines the full $*$-distribution of $\bsb$, referred to as fully matricial Stieltjes transform; it assigns a $(2m\times 2m)$ matrix to each $m\in\N$ and $B\in \bbH_{+}(M_{2m}(\C))$, given by
	\beq\label{eq:Gfn_full}
	((\Brkt{\cdot}\otimes I_{2})\otimes I_{m})(\bsH_{\bsa_{n}}\otimes I_{m}-B)^{-1}\to ((\Brkt{\cdot}\otimes I_{2})\otimes I_{m})(\bsH_{\bsa}\otimes I_{m}-B)^{-1}\in M_{2m}(\C).
	\eeq
	Indeed, recall from \cite[Remark 2.7]{Belinschi-Popa-Vinnikov2012} that $\bsa_{n}\to\bsa$ in $*$-distribution if and only if the fully matricial Stieltjes transform converges.
\end{rem}

\subsection{Free probabilistic inputs}
In what follows, we consider $(M_{2}(\caM),\brkt{\cdot},M_{2}(\C))$ as an $M_{2}(\C)$-valued $W\adj$-probability space (see \cite[Section 2]{Belinschi-Mai-Speicher2017} for details), where $M_{2}(\C)$ is equipped with the usual tracial state $\tr$ in \eqref{eq:M}. The main goal of this section is to record the Schwinger-Dyson equation (see \eqref{eq:Dyson} below) for the Stieltjes transform $\bsM_{\bsa_{t}}$ from \cite{Zhong2021}.
More precisely, we have the following result.
\begin{prop}[{\cite[Theorem 3.8]{Zhong2021}}]\label{prop:Dyson}
	Let $\bsa,\bsx\in\caM$ be as in Definition \ref{defn:model}. Then, for any $z\in\C$, $\eta>0$, and $t\geq0$, the $(2\times 2)$ matrix $\bsM_{\bsa_{t}}(\Theta(z,\eta))$ satisfies 
	\beq\label{eq:Dyson}
	M_{\bsa_{t}}(\Theta(z,\eta))=M_{\bsa}(\Theta(z,\eta)+t\tr M_{\bsa_{t}}(\Theta(z,\eta))).
	\eeq
	Equivalently,
	\begin{equation}
		M_{\bsa_{t}-z}(\ii\eta)=M_{\bsa-z}(\ii\eta+t\tr M_{\bsa_{t}-z}(\ii\eta)).
	\end{equation}
\end{prop}

We remark that \cite[Theorem 3.8]{Zhong2021} concerns elliptic $\bsx$ beyond circular, and Proposition~\ref{prop:Dyson} is a special case of their result. 
Proposition \ref{prop:Dyson} offers something beyond the spectral distribution of $\absv{\bsa+\sqrt{t}\bsx}$, in contrast to the Hermitian case. Notice that both sides of \eqref{eq:Dyson} are $(2\times 2)$ matrices of the form
\beq
\begin{pmatrix} \ii\eta & z \\ \ol{z} & \ii\eta\end{pmatrix},\qquad w\in\C_{+},\eta>0,
\eeq
so that \eqref{eq:Dyson} gives two different equations, one for diagonal and the other for off-diagonal entries. These two equations do not imply one another in general; see \eqref{eq:subor_M}. In fact, the identity between off-diagonal entries played a crucial role in \cite{Zhong2021,Belinschi-Yin-Zhong2022} (see e.g. the proof of \cite[Theorem 4.2]{Zhong2021}).

We will not repeat the proof of Proposition \ref{prop:Dyson}, but only remark that it is a consequence of the following operator-valued subordination result using that $\bsx$ and $\bsa$ are $*$-free.
	\begin{prop}[{\cite[Theorem 2.7]{Belinschi-Mai-Speicher2017}}]\label{prop:subor}
		Let $\bsb,\boldsymbol{c}\in\caM$ be $*$-free. Then $\bsH_{\bsb}$ and $\bsH_{\boldsymbol{c}}$ are free over the canonical sub-algebra $M_{2}(\C)$ of $M_{2}(\caM)$, and there exists a unique pair of Fr\'{e}chet analytic self-maps $\omega_{1},\omega_{2}$ of $\bbH_{+}(M_{2}(\C))$ such that the following hold for each $Z\in\bbH_{+}(M_{2}(\C))$.
		\begin{align}
			\im\omega_{1}(Z)\geq&\im Z\AND \im\omega_{2}(Z)\geq\im Z;	\\
			-\bsM_{\bsb}(\omega_{1}(Z))^{-1}=&\omega_{1}(Z)+\omega_{2}(Z)-Z =-\bsM_{\boldsymbol{c}}(\omega_{2}(Z))^{-1};	\label{eq:subor2}\\
			\bsM_{\bsb}(\omega_{1}(Z))=&\bsM_{\boldsymbol{c}}(\omega_{2}(Z)).\label{eq:subor3}
		\end{align}
		Furthermore, the two quantities in \eqref{eq:subor3} are equal to $\bsM_{\bsb+\boldsymbol{c}}(Z)$.
	\end{prop}
\begin{rem}
	Recall from Remark \ref{rem:full} that the $*$-distribution of an operator $\bsa$ is determined by its fully matricial Stieltjes transform (see \eqref{eq:Gfn_full}). 
	Since Proposition \ref{prop:subor} shows the analytic subordination for only $(2\times2)$-valued Stieltjes transform, the functions $\omega_{1},\omega_{2}$ by themselves determine the Brown measure of $\bsb+\boldsymbol{c}$ but not its $*$-distribution. 
	However, in fact, \cite[Theorem 2.7]{Belinschi-Mai-Speicher2017} covers general operator-valued setting beyond $M_{2}(\C)$, hence shows  fully matricial analytic subordination, i.e. there are natural `lifts' of $\omega_{1},\omega_{2}$ to $\bbH_{+}(M_{2n}(\C))$ for all $n$; see also \cite[Theorem 2.13 and Lemma 2.14]{Belinschi-Popa-Vinnikov2012}.  
	We also remark that, in the Hermitian case, the scalar-valued analytic subordination functions can serve as an alternative definition of free additive convolution; see \cite[Theorem 4.1]{Belinschi-Bercovici2007}.
	Similarly, if we extend $\omega_{1},\omega_{2}$ to the fully matricial setting, they can determine the $*$-distribution of $\bsb+\boldsymbol{c}$ with these functions.
\end{rem}

For later purposes, we define the map $v_{t}:\C\times(0,\infty)\to(0,\infty)$ as
\beq\label{eq:def_v}
v_{t}(z,\eta)\deq \eta+t\im\tr\brkt{G_{\bsa_{t}}(\Theta(z,\eta))}\equiv \eta+t\im\tr M_{\bsa_{t}}(\Theta(z,\eta)).
\eeq
Note that \eqref{eq:Mim}, \eqref{eq:Dyson}, and \eqref{eq:def_v} give for each $z\in\C$ and $\eta>0$ that
\begin{equation}\label{eq:veq}
	M_{\bsa_{t}}(\Theta(z,\eta))=M_{\bsa}(\Theta(z,\eta)+t\im \tr M_{\bsa_{t}}(\Theta(z,\eta)))=M_{\bsa}(\Theta(z,v_{t}(z,\eta))),
\end{equation}
	so that $\ii v_{t}(z,\eta)=\omega_{1}(\Theta(z,\eta))$ where $\omega_{1}$ is the Fr\'{e}chet analytic function in Proposition \ref{prop:subor}. Then it immediately follows that $v_{t}$ is real analytic in $\C\times (0,\infty)$. 
	Writing out the elements of the $(2\times 2)$ matrices in \eqref{eq:veq} using Schur complement formula, 
	we get
	\beq\begin{aligned}\label{eq:subor_M}
		\Brkt{\dfrac{\eta}{\absv{\bsa_{t}-z}^{2}+\eta^{2}}}&=\Brkt{\dfrac{v_{t}(z,\eta)}{\absv{\bsa-z}^{2}+v_{t}(z,\eta)^{2}}},	\\
		\Brkt{\dfrac{1}{\absv{\bsa_{t}-z}^{2}+\eta^{2}}(\bsa_{t}-z)}&=\Brkt{\dfrac{1}{\absv{\bsa-z}^{2}+v_{t}(z,\eta)^{2}}(\bsa-z)}.
	\end{aligned}\eeq
	Combining \eqref{eq:def_v} and \eqref{eq:subor_M}, we find that $v\equiv v_{t}(z,\eta)$ satisfies
	\beq\label{eq:vv}
	v=\eta +t\Brkt{\frac{v}{\absv{\bsa-z}^{2}+v^{2}}}.
	\eeq
	
	While the original definition of $v_{t}$ in \eqref{eq:def_v} is given in terms of the spectral density of $\absv{\bsa_{t}-z}=\absv{\bsa+\sqrt{t}\bsx-z}$ at the origin (regularized by $\eta$), the identity \eqref{eq:vv} may serve as an alternative definition of $v_{t}(z,\eta)$. Indeed, with a few lines of simple computation in \cite{Zhong2021} one can show that \eqref{eq:vv} has a unique positive solution $v$ for each fixed $t,\eta>0$, and $z\in\C$, and it coincides with $v_{t}(z,\eta)$ given in \eqref{eq:def_v}.  Finally, we remark that \eqref{eq:vv} trivially implies $v^{2}-\eta v\leq t$, which shows that $v_{t}(z,\eta)$ is bounded from above as
	\beq\label{eq:v_bdd}
	v_{t}(z,\eta)-\eta\leq t^{-1/2},\qquad z\in\C,\eta>0,t>0.
	\eeq
	\subsection{Regularized Brown measure}
	In this section, we collect preliminary results on the function $v$ and recall the proof of \cite[Theorem 7.10]{Belinschi-Yin-Zhong2022}.
	\begin{prop}[{\cite[Lemmas 3.6 and 3.7]{Zhong2021}}]\label{prop:subor_cont} Let $t>0$ and define $v_{t}:\C\times (0,\infty)\to(0,\infty)$ as in \eqref{eq:def_v}.
		\begin{itemize}
			\item[(i)] For each $z\in\C$, the map $v_{t}(z,\cdot)$ is strictly increasing on $(0,\infty)$.
			\item[(ii)] The map $v_{t}$ extends continuously to $\C\times[0,\infty)$.
			\item[(iii)] The extension $v_{t}(\cdot,0)$ is positive and real analytic in the open set $\caD_{t}$ and identically zero on $\C\setminus\caD_{t}$.
		\end{itemize}
	\end{prop}
	In fact, when $z\in\caD$ we easily see that the equation 
	\beq\label{eq:vvv}
	\Brkt{\frac{1}{\absv{\bsa-z}^{2}+v^{2}}}=\frac{1}{t}
	\eeq
	has a positive solution $v>0$, and this solution is unique and identical to $v_{t}(z,0)$; see \cite[Lemma 3.6]{Zhong2021}.
	
	Next, we recall the following lemma from \cite{Belinschi-Yin-Zhong2022}, which introduces a smoothed density $\rho_{t,\eta}$ of the Brown measure $\rho_{\bsa_{t}}$, indexed by $\eta>0$. The lemma also shows that $\rho_{t,\eta}$ admits an expression involving $v_{t}(z,\eta)$ instead of $\bsx$, and that it is uniformly bounded by $(\pi t)^{-1}$.
	\begin{lem}[{\cite[Lemma 7.11]{Belinschi-Yin-Zhong2022}}]\label{lem:FC_smooth}
		Let $\bsa,\bsx\in\caM$ be as in Definition \ref{defn:model} and $t>0$ be fixed. For each $\eta>0$, the map
		\beq\label{eq:FC_smooth}
		z\mapsto \brkt{\log(\absv{\bsa_{t}-z}^{2}+\eta^{2})}>0
		\eeq
		is $C^{2}$ in $\C\cong\R^{2}$. Defining 
		\beq\label{eq:reg_def}
		\rho_{t,\eta}(z)\deq\frac{1}{4\pi}\Delta_{z}\brkt{\log(\absv{\bsa_{t}-z}^{2}+\eta^{2})}	=\frac{1}{2\pi}\Delta_{z}\brkt{\log\sqrt{\absv{\bsa_{t}-z}^{2}+\eta^{2}}},
		\eeq
		we have
		\begin{multline}\label{eq:Delta_log_conc}
			\pi\rho_{t,\eta}(z)
			=\Brkt{\frac{1}{\absv{\bsa-z}^{2}+v^{2}}\frac{v^{2}}{\absv{\bsa-z}_{*}^{2}+v^{2}}}	\\ +\left(\frac{\eta}{2tv^{3}}+\Brkt{\dfrac{1}{(\absv{\bsa-z}^{2}+v^{2})^{2}}}\right)^{-1}\Absv{\Brkt{\dfrac{1}{(\absv{\bsa-z}^{2}+v^{2})^{2}}(\bsa-z)}}^{2},
		\end{multline}
		where we abbreviated $v\equiv v_{t}(z,\eta)$. Furthermore, $0<\rho_{t,\eta}(z)<(\pi t)^{-1}$ for all $z\in\C$.
	\end{lem}
	One of the most important ideas in the proof of \cite[Lemma 7.11]{Belinschi-Yin-Zhong2022} is the identity
	\begin{equation}\label{eq:log_dz}
		\frac{\dd}{\dd z}\brkt{\log\absv{\bsa_{t}-z}^{2}+\eta^{2}}=\Brkt{(\bsa_{t}-z)\adj\frac{1}{\absv{\bsa_{t}-z}^{2}+\eta^{2}}},
	\end{equation}
	whose right-hand side is an off-diagonal block of $\bsM_{\bsa_{t}-z}(\ii\eta)=\bsM_{\bsa-z}(\ii v_{t}(z,\eta))$. Then one can compute $\rho_{t,\eta}(z)$ by taking the $\ol{z}$-derivative of $\bsM_{\bsa-z}(\ii v_{t}(z,\eta))$, which can be done by differentiating the Dyson equation \eqref{eq:vv}.
	
	On the other hand, by taking the limit $\eta\to0$ in Lemma \ref{lem:FC_smooth}, we can recover the Brown measure $\rho_{\bsa_{t}}$: Indeed, from the definition \eqref{eq:Brown} of the Brown measure and \eqref{eq:reg_def}, we have for each $h\in C_{c}^{2}(\C)$ that
	\begin{multline*}
		\int_{\C}h(z)(\dd\rho_{\bsa_{t}}(z)-\rho_{t,\eta}(z)\dd^{2}z)=\frac{1}{4\pi}\int_{\C}\Delta h(z)\Brkt{\log\frac{\absv{\bsa_{t}-z}^{2}}{\absv{\bsa_{t}-z}^{2}+\eta^{2}}}\dd^{2}z	\\
		=-\frac{1}{2\pi}\int_{\C}\Delta h(z)\int_{0}^{\eta}\Brkt{\frac{y}{\absv{\bsa_{t}-z}^{2}+y^{2}}}\dd y\dd^{2}z
		=-\frac{1}{2\pi}\int_{\C}\Delta h(z)\int_{0}^{\eta}(v_{t}(z,y)-y)\dd y\dd^{2}z,
	\end{multline*}
	where we used the definition of $v_{t}(z,y)$ in \eqref{eq:def_v} in the last equality. Recalling \eqref{eq:v_bdd} we get
	\beq\label{eq:reg_conv}
	\Absv{\int_{\C}h(z)(\dd\rho_{\bsa_{t}}(z)-\rho_{t,\eta}(z)\dd^{2}z)}\leq \frac{\eta}{\sqrt{t}}\norm{\Delta h}_{L^{1}(\C)}.
	\eeq
	A qualitative version of this weak convergence was used in \cite{Belinschi-Yin-Zhong2022} to compute the density of $\rho_{\bsa_{t}}$ in the bulk as a limit of $\rho_{t,\eta}$:	
	\begin{lem}[Theorem 7.10 of \cite{Belinschi-Yin-Zhong2022}]\label{lem:dens}\
		\begin{itemize}
			\item[(i)] The Brown measure $\rho_{\bsa_{t}}$ absolutely continuous on $\C$ with a density $\rho$. 
			\item[(ii)] The density admits the following formula for $z\notin\partial\caD_{t}$: If $z\in\caD_{t}$ we have
			\begin{align}
				\rho(z)&=\frac{1}{\pi}\Brkt{\dfrac{1}{\absv{\bsa-z}^{2}+v^{2}}\dfrac{v^{2}}{\absv{\bsa-z}_{*}^{2}+v^{2}}}	\nonumber\\ \qquad &\qquad\qquad+\Brkt{\dfrac{1}{(\absv{\bsa-z}^{2}+v^{2})^{2}}}^{-1}\Absv{\Brkt{\dfrac{1}{(\absv{\bsa-z}^{2}+v^{2})^{2}}(\bsa-z)}}^{2},\label{eq:dens_bulk}
			\end{align}
			where we abbreviated $v=v_{t}(z,0)$, and if $z\in\C\setminus\ol{\caD}_{t}$ we have $\rho(z)=0$.
			
			\item[(iii)] We have
			\begin{equation}\label{eq:bulk}
				\lim_{\eta\to 0}\rho_{t,\eta}(z)=\rho(z)
			\end{equation}
			uniformly on compact subsets of $\caD_{t}$.
		\end{itemize}
	\end{lem}
	
	\section{Proof of Theorem \ref{thm:sharp}}\label{sec:proof}
	We prove Theorem \ref{thm:sharp} in this section. As such, we always impose Assumption~\ref{assump} on $\bsa$ throughout this section. For simplicity, since we consider fixed $t>0$ in Theorem \ref{thm:sharp}, we take $t=1$ without loss of generality and abbreviate 
	\beq 
	\caD\equiv\caD_{1},\quad  v\equiv v_{1}, \AND \rho_{\eta}\equiv\rho_{1,\eta}.
	\eeq
	We write $\bsa+\bsx$ in place of $\bsa_{t}$ to avoid confusion.
	\begin{lem}\label{lem:v_asymp}
		For each $z_{0}\in\partial\caD$, there exists $\delta>0$ such that the following holds uniformly over $(z,\eta)\in D(z_{0},\delta)\times[0,\delta]$; (recall the relation $\sim_{i}$ from Notational Remark \ref{nrem:asymp})
		\beq\label{eq:v_asymp}
		v(z,\eta)\sim_{(z_{0},\delta)}\begin{cases}
			\dfrac{\eta}{1-f_{\bsa}(z)+\eta^{2/3}}, & f_{\bsa}(z)\leq 1, \\
			\sqrt{f_{\bsa}(z)-1}+\eta^{1/3}, & f_{\bsa}(z)>1.
		\end{cases}
		\eeq
		Furthermore, for $\eta\equiv0$ and as $z\to z_{0}$, we have the more precise asymptotics
		\beq\label{eq:v0_asymp}
		v(z,0)=\begin{cases}
			0, & f_{\bsa}(z)\leq 1,\\
			\Brkt{\dfrac{1}{\absv{\bsa-z}^{4}}}^{-1/2}\sqrt{f_{\bsa}(z)-1}+O(f_{\bsa}(z)-1), & f_{\bsa}(z)>1.
		\end{cases}
		\eeq
	\end{lem}
	\begin{proof}
		First of all, recall $\partial\caD\subset\C\setminus\caD\subset\C\setminus\spec(\bsa)$, which implies that there exists $\delta>0$ such that $z\mapsto (\bsa-z)^{-1}$ is analytic and norm-bounded on the disk $D(z_{0},\delta)$. Note that we have $v(z,\eta)\leq 1$ for all $z\in\C$ and $\eta>0$ due to \eqref{eq:v_bdd}. For all $v\in\R$ and $z$ in the disk $\{z:\absv{z-z_{0}}<\delta\}$, we have
		\beq\label{eq:f}
		\Brkt{\frac{1}{\absv{\bsa-z}^{2}+v^{2}}}-\Brkt{\frac{1}{\absv{\bsa-z}^{2}}}
		=-\Brkt{\frac{v^{2}}{\absv{\bsa-z}^{2}(\absv{\bsa-z}^{2}+v^{2})}}.
		\eeq
		Recall also that $v\equiv v(z,\eta)$ satisfies \eqref{eq:vv}; we always take $v=v(z,\eta)$ with $(z,\eta)\in D(z_{0},\delta)\times [0,\delta]$ in the rest of the proof. Combining \eqref{eq:vv} with \eqref{eq:f} gives
		\beq\label{eq:f_asymp}\begin{aligned}
			\eta=&v-\Brkt{\dfrac{v}{\absv{\bsa-z}^{2}+v^{2}}}	
			=v\left(1-f_{\bsa}(z)+v^{2}\Brkt{\frac{1}{\absv{\bsa-z}^{2}(\absv{\bsa-z}^{2}+v^{2})}}\right).
		\end{aligned}\eeq
		
		We first consider the regime $f_{\bsa}(z)\leq 1$. In this case, we have from \eqref{eq:f_asymp} and $v\leq1$ that
		\beqs
		\eta=v(1-f_{\bsa}(z))+\Brkt{\frac{v^{3}}{\absv{\bsa-z}^{2}(\absv{\bsa-z}^{2}+v^{2})}}\geq v(1-f_{\bsa}(z))\vee v^{3}\Brkt{\frac{1}{\absv{\bsa-z}^{2}(\absv{\bsa-z}^{2}+1)}},
		\eeqs
		which immediately implies 
		\beq\label{eq:out_up}
		v\lesssim_{(z_{0},\delta)}\frac{\eta}{1-f_{\bsa}(z)}\wedge \eta^{1/3}\sim \frac{\eta}{1-f_{\bsa}(z)+\eta^{2/3}}.
		\eeq
		Conversely, we have
		\beq\label{eq:out_lo_1}
		\eta=v(1-f_{\bsa}(z))+\Brkt{\frac{v^{3}}{\absv{\bsa-z}^{2}(\absv{\bsa-z}^{2}+v^{2})}}\leq v(1-f_{\bsa}(z))+v^{3}\Brkt{\frac{1}{\absv{\bsa-z}^{4}}}.
		\eeq
		Plugging $v\lesssim \eta^{1/3}$ from \eqref{eq:out_up} into \eqref{eq:out_lo_1} and using $\norm{(\bsa-z)^{-1}}\lesssim_{(z_{0},\delta)}1$ proves
		\beq\label{eq:out_lo}
		\eta\leq v(1-f_{\bsa}(z))+v^{3}\Brkt{\frac{1}{\absv{\bsa-z}^{4}}}\lesssim_{(z_{0},\delta)} v(1-f_{\bsa}(z)+\eta^{2/3}),
		\eeq
		which is the complementing lower bound for \eqref{eq:out_up}. This proves the first line of \eqref{eq:v_asymp}.
		
		Next, we consider $f_{\bsa}(z)>1$. Here we first establish the lower bound; we use \eqref{eq:f_asymp} to write
		\beq\label{eq:in_lo_1}
		v^{3}\Brkt{\frac{1}{\absv{\bsa-z}^{4}}}\geq v^{3}\Brkt{\frac{1}{\absv{\bsa-z}^{2}(\absv{\bsa-z}^{2}+v^{2})}}=\eta +v(f_{\bsa}(z)-1),
		\eeq
		and using each term on the rightmost side of \eqref{eq:in_lo_1} proves
		\beq\label{eq:in_lo}
		v\gtrsim_{(z_{0},\delta)} \eta^{1/3}+\sqrt{f_{\bsa}(z)-1}.
		\eeq
		Conversely, using \eqref{eq:in_lo} as an input, we have
		\begin{multline}
			v^{2}\Brkt{\frac{1}{\absv{\bsa-z}^{2}(\absv{\bsa-z}^{2}+1)}}\leq v^{2}\Brkt{\frac{1}{\absv{\bsa-z}^{2}(\absv{\bsa-z}^{2}+v^{2})}}	\\
			=\frac{\eta}{v}+f_{\bsa}(z)-1\lesssim_{(z_{0},\delta)} \eta^{2/3}+f_{\bsa}(z)-1,
		\end{multline}
		which proves the upper bound complementing \eqref{eq:in_lo}. This completes the proof of \eqref{eq:v_asymp}.
		
		Finally we prove the last assertion \eqref{eq:v0_asymp}. For $f_{\bsa}(z)\leq 1$ the result immediately follows from \eqref{eq:v_asymp} by taking the limit $\eta\to0$. For $f_{\bsa}(z)>1$, the defining equation \eqref{eq:vv} for $v=v(z,0)$ reduces to
		\beq
		0=1-f_{\bsa}(z)+\Brkt{\frac{v^{2}}{\absv{\bsa-z}^{2}(\absv{\bsa-z}^{2}+v^{2}}}=1-f_{\bsa}(z)+v^{2}\Brkt{\frac{1}{\absv{\bsa-z}^{4}}}+O(v^{4}),
		\eeq
		where we used that $v(z,0)>0$. Recalling from \eqref{eq:v_asymp} that $v\sim\sqrt{f_{\bsa}(z)-1}$, we have
		\beq
		v=\Brkt{\frac{1}{\absv{\bsa-z}^{4}}}^{-1/2}\sqrt{f_{\bsa}(z)-1}+O(v^{2})=\Brkt{\frac{1}{\absv{\bsa-z}^{4}}}^{-1/2}\sqrt{f_{\bsa}(z)-1}+O(f_{\bsa}(z)-1),
		\eeq
		completing the proof of Lemma \ref{lem:v_asymp}.
	\end{proof}
	\begin{rem}\label{rem:local_v}
		As in Remark \ref{rem:sharp_local}, exactly the same proof applies without Assumption \ref{assump} provided $z_{0}\in\partial\caD$ and $z_{0}\notin\spec(\bsa)$. We can also make it quantitative so that the implicit constant in \eqref{eq:v_asymp} can be chosen uniformly over $\bsa$ and $z$ satisfying $\norm{\bsa}\leq C$ and $\norm{(\bsa-z)^{-1}}\leq C$, for each fixed constant $C>0$.
	\end{rem}
	
	
	\begin{proof}[Proof of Theorem \ref{thm:sharp}]
		The first part (i) immediately follows from Lemma \ref{lem:dens}. For the rest of the assertions (ii) -- (iv), we recall from \eqref{eq:dens_bulk} the definition of $\rho$ (for $t=1$) in $\caD$;
		\beq\begin{aligned}\label{eq:rho_in}
			\pi\rho(w)=&v(w,0)^{2}\Brkt{\frac{1}{\absv{\bsa-w}^{2}+v(w,0)^{2}}\frac{1}{\absv{\bsa-w}_{*}^{2}+v(w,0)^{2}}}	\\
			&+\Brkt{\frac{1}{(\absv{\bsa-w}^{2}+v(w,0)^{2})^{2}}}^{-1}\Absv{\Brkt{\frac{1}{(\absv{\bsa-w}^{2}+v(w,0)^{2})^{2}}(\bsa-w)}}^{2}.
		\end{aligned}\eeq
		We next prove part (ii). Recall that $v(\cdot,0)$ is continuous and strictly positive in the open domain $\caD$. Thus for a given $\delta>0$ we have a constant $c>0$ such that
		\beq\label{eq:v_bulk}
		c<v(z,0),\qquad \forall z\in\{z\in\caD:\dist(z,\partial\caD)\geq\delta\},
		\eeq
		where we used that the domain in \eqref{eq:v_bulk} is a compact subset of $\caD$. Indeed, the set $\caD$ (hence $\partial\caD$) is bounded by $\norm{\bsa}+1$, since $\absv{z}>\norm{\bsa}+1$ implies 
		\beq\label{eq:z_bdd}
		f_{\bsa}(z)\leq \norm{(\bsa-z)^{-1}}^{2}\leq \frac{1}{(\absv{z}-\norm{\bsa})^{2}}\leq 1.
		\eeq
		Therefore the first term in \eqref{eq:rho_in} is bounded from below by another constant $c'>0$ depending only on $\bsa$ (and implicitly on $t>0$) as
		\beq\begin{aligned}
			&v(z,0)^{2}\Brkt{\frac{1}{\absv{\bsa-z}^{2}+v(z,0)^{2}}\frac{1}{\absv{\bsa-z}_{*}^{2}+v(z,0)^{2}}}
			\geq c^{2}\frac{1}{(\norm{\bsa-z}^{2}+1)^{2}}\geq c',
		\end{aligned}\eeq
		where we used $v(z,0)\leq 1$ from \eqref{eq:v_bdd}. This completes the proof of (ii).
		
		For (iii) and (iv), we fix a point $z\in\partial\caD$ and (formally) consider the right-hand side of \eqref{eq:rho_in} as a function of two independent variables $w$ and $v\equiv v(w,0)$. Then we expand it with respect to $v$ at $0$, using that $w\mapsto (\bsa-w)^{-1}$ is analytic around $z$. More precisely, we have
		\begin{align}\label{eq:rho_edge_expa}
			\pi\rho(w)
			=&v^{2}\Brkt{\frac{1}{\absv{\bsa-w}^{2}\absv{\bsa-w}_{*}^{2}}}	\\
			&+\left(\Brkt{\absv{\bsa-w}^{-4}}-2v^{2}\brkt{\absv{\bsa-w}^{-6}}\right)^{-1}\Absv{\partial_{\ol{w}}f_{\bsa}(w)-2v^{2}\partial_{\ol{w}}\Brkt{\absv{\bsa-w}^{-4}}}^{2}+O(v^{4}).\nonumber
		\end{align}
		
		Next, for (iii), we apply the following estimates to \eqref{eq:rho_edge_expa}: Uniformly over $\absv{w-z}<\delta$, we have that
		\begin{align}
			\Norm{\frac{1}{\bsa-w}-\frac{1}{\bsa-z}}\lesssim_{(z,\delta)}&\absv{w-z}, \\
			\absv{\partial_{\ol{z}}f_{\bsa}(w)-\partial_{\ol{z}}f_{\bsa}(z)}\lesssim_{(z,\delta)}&\absv{w-z}, \\
			\absv{f_{\bsa}(w)-1}=\absv{f_{\bsa}(w)-f_{\bsa}(z)}\lesssim_{(z,\delta)}&\absv{w-z}, \label{eq:v_edge_sharp0}\\
			v(w,0)\lesssim_{(z,\delta)} \sqrt{(f_{\bsa}(w)-1)_{+}}\lesssim_{(z,\delta)}&\sqrt{\absv{w-z}},\label{eq:v_edge_sharp}
		\end{align}
		where we used Lemma \ref{lem:v_asymp} in \eqref{eq:v_edge_sharp}. As a result, we obtain from \eqref{eq:rho_edge_expa} that (in fact, here we only need \eqref{eq:rho_edge_expa} up to an $O(v^{2})$ error)
		\beq\label{eq:edge_sharp_0}
		\rho(w)=\frac{1}{\pi}\brkt{\absv{\bsa-z}^{-4}}^{-1}\absv{\partial_{\ol{z}}f_{\bsa}(z)}^{2}+O(\absv{w-z}),
		\eeq
		which proves \eqref{eq:edge_sharp}. Since $\absv{\partial_{\ol{z}}f_{\bsa}(z)}=\absv{\brkt{\absv{\bsa-z}^{-4}(\bsa-z)}}\neq0$ away from $\caC$ and \eqref{eq:edge_sharp_0} is valid for $w\in\caD$, combining with Lemma \ref{lem:dens} (ii) leads\footnote{We define $\rho(z)$ by the same formula as in \eqref{eq:dens_bulk} for $z\in\partial\caD$, which is of Lebesgue measure zero under Assumption \ref{assump}; see Remark \ref{rem:assump_conseq}.} to \eqref{eq:edge_sharp_1}. This completes the proof of part (iii).
		
		Finally we prove the higher order expansion (iv). In this case we have $\nabla f_{\bsa}(z)=0$, so that the function $f_{\bsa}$ is locally quadratic around $z$. Consequently the estimates \eqref{eq:v_edge_sharp0} and \eqref{eq:v_edge_sharp} can be improved to
		\beq\begin{aligned}
			\absv{f_{\bsa}(w)-1}\lesssim_{(z,\delta)}& \absv{w-z}^{2}, \\
			v(w,0)-\brkt{\absv{\bsa-z}^{-4}}^{-1/2}\sqrt{(f_{\bsa}(w)-1)_{+}}\lesssim_{(z,\delta)}&\absv{w-z}^{2}.
		\end{aligned}\eeq
		Plugging these into \eqref{eq:rho_edge_expa}, we have
		\beq\begin{aligned}\label{eq:rho_edge_quad}
			\pi\rho(w)=&\brkt{\absv{\bsa-z}^{-4}}^{-1}\brkt{\absv{\bsa-w}^{-2}\absv{\bsa-w}_{*}^{-2}}(f_{\bsa}(w)-1)_{+}\\
			&+(\brkt{\absv{\bsa-w}^{-4}}+O(\absv{w-z}^{2}))^{-1}\left(\absv{\partial_{\ol{z}}f_{\bsa}(w)}+O(\absv{w-z}^{2})\right)^{2}+O(\absv{w-z}^{4})	\\
			=&\brkt{\absv{\bsa-z}^{-4}}^{-1}\brkt{\absv{\bsa-z}^{-2}\absv{\bsa-z}_{*}^{-2}}(f_{\bsa}(w)-1)_{+}	\\
			&+\brkt{\absv{\bsa-z}^{-4}}^{-1}\absv{\partial_{\ol{z}}f_{\bsa}(w)}^{2}+O(\absv{w-z}^{3}),
		\end{aligned}\eeq
		and the asymptotics \eqref{eq:edge_quad} in (iv) follows from substituting Taylor expansions of $f_{\bsa}$ and $\partial_{\ol{z}}f_{\bsa}$ into \eqref{eq:rho_edge_quad}. 
		
		The upper bound in the inequality \eqref{eq:edge_quad_1} follows immediately from the expansion \eqref{eq:edge_quad} since all prefactors in \eqref{eq:edge_quad} and $\norm{H[f_{\bsa}](z)}$ are bounded as $z\notin\spec(\bsa)$. For the lower bound, recall $\Delta f_{\bsa}(z)=\Tr H[f_{\bsa}](z)=4\brkt{\absv{(\bsa-z)^{-2}}^{2}}>0$ from \eqref{eq:subhar}, 
		so that the matrix $H[f_{\bsa}](z)\in\R^{2\times 2}$ has at least one positive eigenvalue. We take $\lambda\neq 0$ to be the smallest (in modulus) nonzero eigenvalue of $H[f_{\bsa}](z)$. Recalling that $P_{z}$ is the orthogonal projection onto the null space of $H[f_{\bsa}](z)$, whenever $(w-z)\in\caS(z,\kappa)$ we have
		\beq\label{eq:edge_quad_lo_1}
		\Norm{H[f_{\bsa}](z)\begin{pmatrix}\re[w-z]\\ \im[w-z]\end{pmatrix}}
		\geq \absv{\lambda}\Norm{(I-P_{z})\begin{pmatrix}\re[w-z]\\ \im[w-z]\end{pmatrix}}
		\geq \absv{\lambda}\absv{w-z}.
		\eeq
		Substituting \eqref{eq:edge_quad_lo_1} into the asymptotics \eqref{eq:edge_quad} proves that
		\beq
		\rho(w)\geq c\absv{w-z}^{2}+O(\absv{w-z}^{3}),
		\eeq
		which gives lower bound in \eqref{eq:edge_quad_1}. This completes the proof of Theorem \ref{thm:sharp}.
	\end{proof}
	We conclude this section with a side result about the limit of the regularized Brown measure exactly at the edge,  complementing the similar result in the bulk \eqref{eq:dens_bulk}-\eqref{eq:bulk}: 
	\begin{lem}\label{lem:dens1}
		If $z\in\partial\caD$, we have 
		\begin{equation}\label{eq:dens_sharp_edge}
			\lim_{\eta\to0}\rho_{\eta}(z)=\frac{2}{3}\Brkt{\frac{1}{\absv{\bsa-z}^{4}}}^{-1}\Absv{\Brkt{\frac{1}{\absv{\bsa-z}^{4}}(\bsa-z)}}^{2}.
		\end{equation}
	\end{lem}
	\begin{proof}
		%
		We temporarily abbreviate $v(\eta)\equiv v(z,\eta)$ along the proof since we consider fixed $z\in\partial\caD$. By \eqref{eq:Delta_log_conc} and \eqref{eq:v0_asymp}, it suffices to prove that
		\begin{multline}\label{eq:sharp_out_2}
			\lim_{\eta\to 0}\left(\frac{\eta}{2v(\eta)^{3}}+\Brkt{\dfrac{1}{(\absv{\bsa-z}^{2}+v(\eta)^{2})^{2}}}\right)^{-1}\Absv{\Brkt{\dfrac{1}{(\absv{\bsa-z}^{2}+v(\eta)^{2})^{2}}(\bsa-z)}}^{2}	\\
			=\frac{2}{3}\Brkt{\dfrac{1}{\absv{\bsa-z}^{4}}}^{-1}\Absv{\Brkt{\dfrac{1}{\absv{\bsa-z}^{4}}(\bsa-z)}}^{2}.
		\end{multline}
		Recall from Assumption \ref{assump} that $z\notin{\caD}$ implies $z\notin \spec(\bsa)$, hence the operator $(\bsa-z)^{-1}$ is bounded. Note also that $z\in \C\setminus\caD$ implies $v(0)=0$ by Proposition \ref{prop:subor_cont}.
		
		For the first term in the denominator of \eqref{eq:sharp_out_2}, we use \eqref{eq:f_asymp} and $v(\eta)\to0$ as $\eta\to0$ to obtain 
		\beq\label{eq:sharp_out_crit}\begin{aligned}
			\frac{\eta}{v(\eta)^{3}}=&\frac{1-f_{\bsa}(z)}{v(\eta)^{2}}+\Brkt{\dfrac{1}{\absv{\bsa-z}^{2}(\absv{\bsa-z}^{2}+v(\eta)^{2})}}
			\to
			\brkt{\absv{\bsa-z}^{-4}},
		\end{aligned}\eeq
		where we used that $(\bsa-z)^{-1}$ is bounded when $f_{\bsa}(z)=1$. For the remaining quantities in \eqref{eq:sharp_out_2}, we have the norm (in $\caM$) convergences
		\beq\label{eq:sharp_out_crit_1}
		\lim_{\eta\to0}\frac{1}{(\absv{\bsa-z}^{2}+v(\eta)^{2})^{2}}
		=\frac{1}{\absv{\bsa-z}^{4}},\qquad 
		\lim_{\eta\to0}\dfrac{1}{(\absv{\bsa-z}^{2}+v(\eta)^{2})^{2}}(\bsa-z)=\frac{1}{\absv{\bsa-z}^{4}}(\bsa-z).
		\eeq
		Finally, combining \eqref{eq:bd=level}, \eqref{eq:sharp_out_crit}, and \eqref{eq:sharp_out_crit_1} proves \eqref{eq:sharp_out_2}. This completes the proof of Lemma \ref{lem:dens}.
	\end{proof}
	
	\begin{rem}
		Along the proof of Theorem \ref{thm:sharp}, in \eqref{eq:edge_sharp_0} and \eqref{eq:dens_sharp_edge}, we have proved that for $z\in\partial\caD\setminus\caC$
		\beq\begin{aligned}\label{eq:rho_unstable}
			\lim_{\eta\to0}\lim_{w\to z}\rho_{\eta}(w)=&\frac{2}{3\pi}\Brkt{\frac{1}{\absv{\bsa-z}^{4}}}^{-1}\Absv{\Brkt{\frac{1}{\absv{\bsa-z}^{4}}(\bsa-z)}}^{2},	\\
			\lim_{w\to z,w\in\caD}\lim_{\eta\to0}\rho_{\eta}(w)=&\frac{1}{\pi}\Brkt{\frac{1}{\absv{\bsa-z}^{4}}}^{-1}\Absv{\Brkt{\frac{1}{\absv{\bsa-z}^{4}}(\bsa-z)}}^{2}.
		\end{aligned}\eeq
		Notice the discrepancy between the two limits in \eqref{eq:rho_unstable}; the former is strictly smaller than the latter by a factor of $2/3$. This shows that the smoothed density $\rho_{\eta}$ is highly unstable around the edge, in the sense that $\lim_{(w,\eta)\to(z,0)}\rho_{\eta}(w)$ depends on the direction $(w,\eta)$ approaches $(z,0)$. Upon close inspection of the proof, we find that the limit \eqref{eq:sharp_out_crit} is responsible for this instability, that is, the limit
		\beq
		\lim_{(w,\eta)\to(z,0)}\frac{\eta}{v(w,\eta)^{3}}
		\eeq
		depends (if exists) on the limit direction; see \eqref{eq:v_asymp} for an explanation. Nonetheless, plugging in 
		\beq
		\liminf_{(w,\eta)\to(z,0)}\frac{\eta}{v(w,\eta)^{3}}\geq 0
		\eeq
		to the definition of $\rho_{\eta}(w)$ proves that the limit along $\eta=0$ is the largest, that is,
		\beq
		\limsup_{(w,\eta)\to(z,0)}\rho_{\eta}(w)=\frac{1}{\pi}\brkt{\absv{\bsa-z}^{-4}}^{-1}\absv{\partial_{\ol{z}}f_{\bsa}(z)}^{2}.
		\eeq
	\end{rem}
	
	\subsection*{Acknowledgments}
	The authors would like to thank Ping Zhong for pointing out further references \cite{Bercovici-Wang-Zhong2023,Bordenave-Capitaine2016} and providing helpful comments.
	
	\appendix

	\section{Irregular edges}\label{append:irreg}
	
	\subsection{Proof of Theorem \ref{thm:higher}}\label{append:higher}
	\begin{proof}[Proof of Theorem \ref{thm:higher}]
		Since the proof is mostly computational we only show the main steps. Denote $z=E+\ii E'\in[0,1]^{2}$, where $z$ is in the domain in Theorem \ref{thm:higher}, i.e.
		\beq
		cE\leq E'\leq c^{-1}E
		\eeq
		for some constant $c>0$. For another suitable constant $C>1$, we divide the support $[0,1]^{2}$ of $\rho_{\bse}$ into $\cup_{j=1}^{4}I_{j}$, where
		\beq\label{eq:Iij}\begin{aligned}
			I_{1}\deq &[0,CE]\times [0,C^{-1}E']\cup [0,C^{-1}E]\times [0,CE'], \\
			I_{2}\deq &[C^{-1}E,CE]\times [C^{-1}E',CE'],	\\
			I_{3}\deq &[0,CE]\times [CE',1]\cup [CE,1]\times [0,CE']=:I_{31}\cup I_{32}, \\
			I_{4}\deq &[CE,1]\times [CE',1].
		\end{aligned}\eeq
		All asymptotic notations below depend on the choice of $c,C,t$, and we omit the subscript $(c,C,t)$ for brevity.
		
		In the first step, we show that $v\equiv v_{t}(z)$, the solution of
		\beq\label{eq:v_irreg}
		(p+1)(q+1)\int_{0}^{1}\int_{0}^{1}\frac{x^{p}y^{q}}{(x-E)^{2}+(y-E')^{2}+v^{2}}\dd x\dd y=\frac{1}{t},
		\eeq
		is exponentially small in $E$, more precisely that
		\beq\label{eq:v_exp}
		-\log v\sim E^{-(p+q)}.
		\eeq
		The proof of \eqref{eq:v_exp} follows from evaluating the integral in \eqref{eq:v_irreg} over each domain $I_{1},\cdots I_{4}$. We easily find that the integral over $I_{2}$ dominates others, which is computed as 
		\beq\begin{aligned}\label{eq:I22}
			\int_{I_{2}}\frac{x^{p}y^{q}}{(x-E)^{2}+(y-E')^{2}+v^{2}}\dd x \dd y\sim E^{p+q}\int_{[-1+c,C-1]^{2}} \frac{1}{\wt{x}^{2}+\wt{y}^{2}+(v/E)^{2}}\dd\wt{x}\dd\wt{y}	\\
			\sim E^{p+q}\int_{\D}\frac{1}{\absv{z}^{2}+(v/E)^{2}}\dd^{2}z\sim E^{p+q}\log\left(\frac{v^{2}}{E^{2}}\right),
		\end{aligned}\eeq
		where we used the change of variables $E\wt{x}=x-E$ and $E'\wt{y}=y-E'$. Since we fixed $t>0$, \eqref{eq:v_exp} immediately follows.
		
		Now we apply \eqref{eq:v_exp} to each term in the formula \eqref{eq:dens_bulk} for the density $\rho$ of $\rho_{\bse+\sqrt{t}\bsx}$. For the first term therein, we write
		\beq\label{eq:dens_irreg_1}
		\Brkt{\frac{1}{\absv{\bse-z}^{2}+v^{2}}\frac{1}{\absv{\bse-z}_{*}^{2}+v^{2}}}=\int_{0}^{1}\int_{0}^{1}\frac{(p+1)(q+1)x^{p}y^{q}}{((x-E)^{2}+(y-E')^{2}+v^{2})^{2}}\dd x\dd y,
		\eeq
		where in the first equality we used that $\bse$, hence $\bse-z$, are normal so that $\absv{\bse-z}=\absv{\bse-z}_{*}$.
		We divide the domain of integration on the right-hand side of \eqref{eq:dens_irreg_1} as in \eqref{eq:Iij}. Again it is easy to see that the integral over $I_{2}$ is dominating, which leads to
		\beq\label{eq:rho_irreg_1}
		v^{2}\Brkt{\frac{1}{\absv{\bse-z}^{2}+v^{2}}\frac{1}{\absv{\bse-z}_{*}^{2}+v^{2}}}\sim E^{p+q}.
		\eeq

		Applying the same argument to the second term in \eqref{eq:dens_bulk}, we find that 
		\begin{multline}
			\Absv{\Brkt{\frac{1}{(\absv{\bse-z}^{2}+v^{2})^{2}}(\bse-z)}}	\\
			\leq(p+1)(q+1)\int_{[0,1]^{2}}\frac{x^{p}y^{q}(\absv{x-E}+\absv{y-E'})}{((x-E)^{2}+(y-E')^{2}+v^{2})^{2}}\dd x\dd y\lesssim E^{p+q}v^{-1},
		\end{multline}
		which, together with \eqref{eq:rho_irreg_1}, gives
		\begin{multline}\label{eq:rho_irreg_2}
			\Brkt{\dfrac{1}{(\absv{\bse-z}^{2}+v^{2})^{2}}}^{-1}\Absv{\Brkt{\dfrac{1}{(\absv{\bse-z}^{2}+v^{2})^{2}}(\bse-z)}}^{2}	\\
			\lesssim (E^{p+q}v^{-2})^{-1}\cdot(E^{p+q}v^{-1})^{2}\sim E^{p+q}.
		\end{multline}
		Plugging in \eqref{eq:rho_irreg_1} and \eqref{eq:rho_irreg_2} to \eqref{eq:dens_bulk}, we conclude
		\beq
		\rho(z)\sim E^{p+q}\sim \absv{z}^{p+q}
		\eeq
		as desired. This concludes the proof of Theorem \ref{thm:higher}.
	\end{proof}
	\subsection{Proof of Proposition \ref{prop:Jacobi}}\label{append:Jacobi}
	In this section, we prove Proposition \ref{prop:Jacobi}. Our proof is inspired by \cite[Section A]{Lee-Schnelli2013}.
	\begin{proof}[Proof of Proposition \ref{prop:Jacobi}]
		We start with the proof of (iii). By Proposition \ref{prop:subor_cont} (ii), we only need to prove that $\caD_{t}\cap \R=(x_{t,-},x_{t,+})$. Clearly $(0,1)\subset\caD_{t}$, since $f_{\bsd}(x)\equiv\infty$ for $x\in(0,1)$. If $x_{t,-}<0$, we also have $(x_{t,-},0]\subset \caD_{t}$ by the definition of  $x_{t,-}$. Applying the same argument to the upper edge, we have $(x_{t,-},x_{t,+})\subset\caD_{t}$. Conversely, if $x<x_{t,-}$ or $x>x_{t,+}$, we immediately have $f_{\bsd}(x)\leq \frac{1}{t}$ by the definition of $x_{t,\pm}$. This proves $\R\setminus (x_{t,-},x_{t,+})\subset \R\setminus\caD_{t}$, completing the proof of (iii).
		
		We next prove (iv), i.e. the solution $v_{t}(\cdot,0)$ of the equation \eqref{eq:vv} satisfies, for some small $\delta>0$, that
		\beq\label{eq:v_higher}
		v=v(x)\equiv v_{t}(x,0)\sim_{(t,\delta)} x^{p_{-}},\qquad x\in[0,\delta].
		\eeq
		Along the proof of \eqref{eq:v_higher} we consider $t<t_{-}$ and $\delta>0$ to be fixed, henceforth omit the subscript $(t,\delta)$ from all asymptotic notations.
		Recall first that $v(0,0)=0$ (since $f_{\bsd}(0)=1/t_{-}<1/t$ by assumption) and that $v(x)$ is continuous, hence we may take $\delta>0$ small enough so that $v(\cdot,0)$ is small in $D(0,\delta)$. Fix a threshold $\epsilon>0$ as
		\beq
		\delta<\epsilon<\frac{1}{100}\wedge \frac{1}{5}\left(1-\frac{t}{t_{-}}\right).
		\eeq
		
		We next show that $v<x/2$. Suppose $v\geq x/2$ on the contrary and rewrite the equation \eqref{eq:vvv} as
		\beq\label{eq:v1}
		\frac{1}{t}=\left(\int_{0}^{\epsilon^{-1}x}+\int_{\epsilon^{-1}x}^{1}\right)\frac{1}{(E-x)^{2}+v^{2}}\rho_{\bsd}(E)\dd E.
		\eeq
		The first integral in \eqref{eq:v1} is estimated as
		\beq\label{eq:v1_1}
		\int_{0}^{\epsilon^{-1}x}\frac{\rho_{\bsd}(E)}{(E-x)^{2}+v^{2}}\dd E
		\leq\int_{0}^{\epsilon^{-1}x}\frac{\rho_{\bsd}(E)}{v^{2}}\dd E
		\lesssim\frac{x^{p_{-}+1}}{v^{2}}\leq x^{p_{-}-1},
		\eeq
		where we used the assumption $v\geq x/2$ in the last inequality. The second integral in \eqref{eq:v1} is bounded from above by
		\beq\label{eq:v1_2}
		\int_{\epsilon^{-1}x}^{1}\frac{\rho_{\bsd}(E)}{(E-x)^{2}}\dd E 
		\leq \int_{\epsilon^{-1}x}^{1}\frac{\rho_{\bsd}(E)}{E^{2}(1-\epsilon)^{2}}\dd E 
		\leq \frac{1}{t_{-}(1-\epsilon)^{2}}\leq\frac{1+3\epsilon}{t_{-}}<\frac{1}{t}-c,
		\eeq
		where the last two inequalities follows from the definition of $\epsilon$. Thus we conclude from \eqref{eq:v1} that
		\beq
		c\leq x^{p_{-}-1},
		\eeq
		which is a contradiction by taking small $\delta>0$ since $p_{-}>1$. This proves $v<x/2$.
		
		Now we show the upper bound in \eqref{eq:v_higher}, i.e. $v\lesssim x^{p_{-}}$. Since $v<x/2$ from the previous paragraph, we have 
		\beqs
		0<x-v<x+v<\epsilon^{-1}x<1
		\eeqs
		Thus we may rewrite \eqref{eq:vvv} as
		\beq\label{eq:v2}
		\frac{v}{t}=\left(\int_{0}^{x-v}+\int_{x-v}^{x+v}+\int_{x+v}^{\epsilon^{-1}x}+\int_{\epsilon^{-1}x}^{1}\right)\frac{v}{(E-x)^{2}+v^{2}}\rho_{\bsd}(E)\dd E	=:(A)+(B)+(C)+(D).
		\eeq
		We estimate each integral in \eqref{eq:v2} as follows;
		\beq
		(A)\lesssim \int_{0}^{x-v}\frac{vE^{p_{-}}}{(E-x)^{2}}\dd E \leq vx^{p_{-}}\int_{0}^{x-v}\frac{1}{(E-x)^{2}}\dd E\leq x^{p_{-}};
		\eeq
		\beq
		(B)\sim \int_{x-v}^{x+v}\frac{1}{v}E^{p_{-}}\dd E\sim x^{p_{-}};
		\eeq
		\beq
		(C)\lesssim \int_{x+v}^{\epsilon^{-1}x}\frac{v}{(E-x)^{2}}E^{p_{-}}\dd E\lesssim x^{p_{-}};
		\eeq
		For the last integral, we use \eqref{eq:v1_2} to get
		\beq
		(D)<\left(\frac{1}{t}-c\right)v.
		\eeq
		Combining all estimates, we find $v\lesssim x^{p_{-}}$ as desired. The lower bound $v\gtrsim x^{p_{-}}$ follows from $(B)\sim x^{p_{-}}$. This completes the proof of (iv). On the other hand Proposition \ref{prop:Jacobi} (v) follows from  Remark \ref{rem:local_v}, since $t>t_{-}$ implies $x_{t,-}<0$, which in turn gives $f_{\bsd}'(x_{t,-})\sim1$ hence $(f_{\bsd}(x)-1)\sim(x-x_{t,-})$.
		
		Finally, we prove (vi), i.e. that the density of $\rho_{\bsd+\sqrt{t}\bsx}$ is bounded from below on $\caD_{t}$. Recalling from part (ii) that the density is constant along vertical segments in $\caD_{t}$, it suffices to show a uniform lower  bound of the density for $z$ in the interval $(x_{t,-},x_{t,+})$. For a small parameter $\delta_{1}>0$, we further divide this interval into
		\beq
		(x_{t,-},x_{t,+})=I_{t,1}\cup I_{t,2}
		\eeq
		where 
		\beq\begin{aligned}
			I_{t,1}\equiv I_{t,1}(\delta_{1})\deq&[x_{t,-}+\delta_{1},x_{t,+}-\delta_{1}]	\\
			I_{t,2}\equiv I_{t,2}(\delta_{1})\deq&(x_{t,-},x_{t,-}+\delta_{1})\cup (x_{t,+}-\delta_{1},x_{t,+}).
		\end{aligned}\eeq
		The first domain $I_{t,1}$ is a compact subset of $\caD_{t}$ so that the density is bounded from below by a constant therein by Theorem \ref{thm:prev} (ii). For the second domain $I_{t,2}$, without loss of generality we focus only on the left portion of $I_{t,2}$, i.e. $[x_{-},x_{-}+\delta_{1})$. If $t>t_{-}$ so that $x_{t,-}<0$, the result again follows from Theorem \ref{thm:sharp} (iii) since $f_{\bsd}'(x_{t,-})\sim1$,
		so that the density is proportional to $\lone_{\ol{\caD}_{t}}$ around $x_{t,-}$. On the other hand if $t<t_{-}$, recall from \eqref{eq:dens_bulk} that the density $\rho(x)$ of $\rho_{\bsd+\sqrt{t}\bsx}$ at $0<x\ll 1$ admits the lower bound
		\beq
		\rho(x)\gtrsim \Brkt{\frac{v^{2}}{\absv{\bsd-x}^{4}+v^{4}}}=\int_{0}^{1}\frac{v^{2}}{\absv{E-x}^{4}+v^{4}}\rho_{\bsd}(E)\dd E.
		\eeq
		Using $v\sim x^{p_{-}}$ from Proposition \ref{prop:Jacobi} (iv) with $p_{-}>1$ (since $t<t_{-}$) and $\rho(E)\sim E^{p_{-}}$ as $E\searrow 0$, we have
		\begin{multline}
			\int_{0}^{1}\frac{v^{2}}{\absv{E-x}^{4}+v^{4}}\rho_{\bsd}(E)
			\gtrsim \int_{0}^{1}\frac{x^{2p_{-}}}{\absv{E-x}^{4}+x^{4p_{-}}}E^{p_{-}}\dd E	\\
			\geq \int_{x-x^{p_{-}}}^{x+x^{p_{-}}}\frac{x^{2p_{-}}}{\absv{E-x}^{4}+x^{4p_{-}}}E^{p_{-}}\dd E\sim 1,
		\end{multline}
		as $x\searrow 0$. By taking $\delta_{1}$ small enough, we have proved that $\rho$ is bounded from below by a constant in $\caD_{t,3}$. This completes the proof of Proposition \ref{prop:Jacobi} (vi).
	\end{proof}
	
	\section{Proof of Theorem \ref{thm:conn}}\label{append:conn}
	
	\begin{proof}[Proof of Theorem \ref{thm:conn}]
		The proof is by contradiction. Take $0<t_{-}<t_{+}$ and assume that the number of components of $\ol{\caD}_{t_{+}}\cup\spec(\bsa)$ is larger than that of $\ol{\caD}_{t_{-}}\cup\spec(\bsa)$. Recall that $\ol{\caD}_{t}\cup\spec(\bsa)$ is closed for each $t\geq0$ and its components are defined with respect to the relative topology inherited from $\C$.
		
		The open set $\caD_{t}$ is bounded for each $t$ since
		\beq
		\absv{z}^{2}\geq 2(\norm{\bsa}^{2}+t)\Longrightarrow \absv{\bsa-z}^{2}\geq t \Longrightarrow z\notin \caD_{t},
		\eeq
		by Cauchy-Schwarz inequality. Thus we may then write the components $C_{i}$ and $D_{j}$ as
		\beq
		\ol{\caD}_{t_{-}}\cup\spec(\bsa)=\bigcup_{i=1}^{k}C_{i},\qquad \ol{\caD}_{t_{+}}\cup\spec(\bsa)=\bigcup_{j=1}^{m}D_{j},\qquad k<m,
		\eeq
		where we allow $m=\infty$ (still $k<\infty$ by hypothesis).
		Note that each $C_{i}$ is closed and open in the relative topology of $\ol{\caD}_{t_{-}}\cup\spec(\bsa)$, hence closed in $\C$. Similarly $D_{j}$'s are closed in $\C$, while they might not be open in $\ol{\caD}_{t_{+}}\cup\spec(\bsa)$ when $m=\infty$. To sum up, each of $\{C_{i}\}$ and $\{D_{j}\}$ is a collection of disjoint, connected, compact subsets of $\C$. 
		
		Then, since $\ol{\caD}_{t_{-}}\subset\ol{\caD}_{t_{+}}$, for each $i\in\{1,\cdots,k\}$ there exists $j_{i}\in\{1,\cdots,m\}$ such that
		\beq
		C_{i}\subset D_{j_{i}}, \qquad C_{i}\cap D_{j}=\emptyset,\,\, \forall j\neq j_{i}.
		\eeq
		Without loss of generality we assume $\{j_{1},\cdots,j_{k}\}=\{1,\cdots,k'\}$ with $k'\leq k$. Notice for each $j>k'$ that the function $f_{\bsa}$ is real analytic in a neighborhood of $D_{j}$, and that $f_{\bsa}(z)\geq 1/t_{+}$ for all $z\in D_{j}$, due to
		\beq
		D_{j}\subset \ol{\caD}_{t_{+}}\setminus\bigcup_{i=1}^{k}C_{i}\subset \ol{\caD}_{t_{+}}\cap(\C\setminus\spec(\bsa)).
		\eeq
		
		In what follows, we deduce that $f_{\bsa}$ has a local maximum in $\C\setminus\spec(\bsa)$, that is, there exists a point $z_{0}\in\C\setminus\spec(\bsa)$ and $r_{0}>0$ so that $f_{\bsa}(z_{0})\geq f_{\bsa}(z)$ for all $\absv{z-z_{0}}<r_{0}$. This would lead to a contradiction since $f_{\bsa}$ is a non-constant subharmonic function in the open set $\C\setminus\spec(\bsa)$, due to $\Delta f_{\bsa}(z)>0$ from \eqref{eq:subhar}.
		
		To this end, we divide the proof into two cases; firstly, we assume that $f_{\bsa}$ is identically $1/t_{+}$ on $\bigcup_{j>k'}D_{j}$. In this case, we take $z_{0}$ to be any point in $\bigcup_{j>k'}D_{j}$ and small enough $r_{0}$ so that the disk $D(z_{0},r_{0})$ lies in $\C\setminus\spec(\bsa)$ while not intersecting with $\bigcup_{j\leq k'}D_{j}$. Then it immediately follows that
		\beq
		D(z_{0},r_{0})\cap \ol{\caD}_{t_{+}}\subset \bigcup_{j>k'}D_{j},
		\eeq
		which in turn implies $f_{\bsa}(w)\leq t_{+}^{-1}=f_{\bsa}(z_{0})$ for any $w\in D(z_{0},r_{0})$.
		
		Secondly, we consider the complementary case when there exists $j_{0}>k'$ and $z\in D_{j_{0}}$ with $f_{\bsa}(z)>t_{+}^{-1}$. In this case we take $z_{0}$ to be the point in $D_{j_{0}}$ with $f_{\bsa}(z_{0})=\max_{z\in D_{j_{0}}}f_{\bsa}(z)$. Then, by the continuity of $f_{\bsa}$, we may take a small enough $r_{0}>0$ so that
		\beq
		D(z_{0},r_{0})\subset \caD_{t_{+}}\cap(\C\setminus\spec(\bsa)),
		\eeq
		which implies that $f_{\bsa}$ attains a local maximum at $z_{0}$ in $D(z_{0},r_{0})$. This completes the proof of Theorem~\ref{thm:conn}.
	\end{proof}
	
\subsection*{Acknowledgment}
		We thank Ping Zhong for pointing out missing references and providing helpful comments. We also thank the anonymous referee for many valuable comments and proposals to streamline the presentation. This work was partially supported by ERC Advanced Grant "RMTBeyond" No. 101020331.


	
	
	
	
	
	
	
	
\end{document}